\documentclass[reqno, 11pt]{amsart}
\usepackage[letterpaper, margin=1in]{geometry}
\usepackage{setspace}

\usepackage[utf8]{inputenc} 
\usepackage[T1]{fontenc}

\usepackage[shortlabels]{enumitem}
\usepackage{amsmath}
\usepackage{amsthm} 
\usepackage{amssymb}
\usepackage{float}
\usepackage{todonotes}
\usepackage{xcolor}
\usepackage{graphicx}
\usepackage{tikz}
\usetikzlibrary{decorations.pathreplacing}
\usepackage{bm}
\usepackage{booktabs}
\usepackage{mathtools}

\usepackage[hidelinks,backref=page]{hyperref}

\usepackage[noabbrev,capitalize,nameinlink]{cleveref}
\crefformat{equation}{(#2#1#3)}
\crefmultiformat{equation}{(#2#1#3)}{ and~(#2#1#3)}{, (#2#1#3)}{ and~(#2#1#3)}

\newtheorem{theorem}{Theorem}[section]
\newtheorem{proposition}[theorem]{Proposition}
\newtheorem{lemma}[theorem]{Lemma}

\newtheorem{corollary}[theorem]{Corollary}
\newtheorem{conjecture}[theorem]{Conjecture}

\theoremstyle{definition}
\newtheorem{definition}[theorem]{Definition}

\newtheorem{problem}[theorem]{Problem}

\newtheorem{question}[theorem]{Question}
\newtheorem{example}[theorem]{Example}

\theoremstyle{remark}
\newtheorem{remark}[theorem]{Remark}

\newcommand{\abs}[1]{\left\lvert#1\right\rvert}

\newcommand{\ang}[1]{\left\langle #1 \right\rangle}

\newcommand{\floor}[1]{\left\lfloor #1 \right\rfloor}

\newcommand{\paren}[1]{\left( #1 \right)}

\newcommand{\set}[1]{\left\{ #1 \right\}}

\newcommand{\CC}{\mathbb{C}}

\newcommand{\FF}{\mathbb{F}}
\newcommand{\RR}{\mathbb{R}}
\newcommand{\PP}{\mathbb{P}}
\newcommand{\NN}{\mathbb{N}}
\newcommand{\ZZ}{\mathbb{Z}}
\newcommand{\QQ}{\mathbb{Q}}

\newcommand{\legendre}[2]{\left(\frac{#1}{#2}\right)}

\title{Uniacute Spherical Codes}
\author[Lepsveridze]{Saba Lepsveridze}
\author[Saatashvili]{Aleksandre Saatashvili}
\author[Zhao]{Yufei Zhao}

\thanks{Lepsveridze and Saatashvili were supported in part by MIT UROP. 
Zhao was supported in part by NSF CAREER award DMS-2044606.}

\address{Lepsveridze, Zhao: Massachusetts Institute of Technology, Cambridge, MA, USA}
\email{\{sabal,yufeiz\}@mit.edu}
\address{Saatashvili: Carnegie Mellon University, Pittsburgh, PA, USA}
\email{asaatash@andrew.cmu.edu}

\begin{document} 

\begin{abstract}
A spherical $L$-code, where $L \subseteq [-1,\infty)$,
consists of unit vectors in $\mathbb{R}^d$ whose pairwise inner products are contained in $L$. 
Determining the maximum cardinality $N_L(d)$ of an $L$-code in $\mathbb{R}^d$ is a fundamental question in discrete geometry and has been extensively investigated for various choices of $L$.
Our understanding in high dimensions is generally quite poor. 
Equiangular lines, corresponding to $L = \{-\alpha, \alpha\}$, is a rare and notable solved case.

Bukh studied an extension of equiangular lines and showed that $N_L(d) = O_L(d)$ for $L = [-1, -\beta] \cup \{\alpha\}$ with $\alpha,\beta > 0$ (we call such $L$-codes ``uniacute''), leaving open the question of determining the leading constant factor.
Balla, Dr\"{a}xler, Keevash, and Sudakov proved a ``uniform bound'' showing $\limsup_{d\to\infty} N_L(d)/d \le 2p$ for $L = [-1, -\beta] \cup \{\alpha\}$ and $p = \lfloor \alpha/\beta \rfloor + 1$.
For which $(\alpha,\beta)$ is this uniform bound tight?
We completely answer this question.

We develop a framework for studying uniacute codes, including a global structure theorem showing that the Gram matrix has an approximate $p$-block structure. 
We also formulate a notion of ``modular codes,'' which we conjecture to be optimal in high dimensions.
\end{abstract}

\maketitle

\section{Introduction}

A \emph{spherical $L$-code}, where $L \subseteq [-1,1)$, consists of unit vectors in $\RR^d$ whose pairwise inner products are contained in $L$.
Let $N_L(d)$ denote the size of the largest spherical $L$-code in $\RR^d$. 
First introduced by Delsarte, Goethals, and Seidel \cite{DGS77}, determining $N_L(d)$ is a central problem in discrete geometry. 

The classical spherical code problem of determining the maximum number of points on a sphere pairwise separated by angle at least $\theta$ is equivalent to $N_L(d)$ with $L = [-1, \cos\theta]$. When $0 < \theta < \pi/2$, we know that $N_L(d)$ grows exponentially in $d$, although there is an exponential gap between the best upper and lower bounds~\cite{KL78,JJP18}.

Another notable case is equiangular lines, which are lines through the origin with pairwise equal angles. For a given common angle $\theta$, this problem corresponds to $L = \{\pm \cos \theta\}$. The problem of determining $N_{\set{\pm \alpha}}(d)$ for a fixed $\alpha \in (0,1)$ and large $d$ was initiated by Lemmens and Seidel \cite{LS73} in 1973, which led to substantial work~\cite{Neu89,Buk16,BDKS18,JP20} that culminated in a recent solution by Jiang, Tidor, Yao, Zhang, and Zhao \cite{JTYZZ21}. 

Bukh~\cite{Buk16} proved that $N_L(d) = O_L(d)$ when $L = [-1, -\beta] \cup \{\alpha\}$ with $\alpha, \beta \in (0,1)$. Here the subscript in $O_L(\cdot)$ means that the hidden constant is allowed to depend on $L$.
Intuitively, Bukh's result says that the allowed acute angles matter more than the allowed obtuse angles.
When $L \subseteq [-1, -\beta] \cup \{\alpha\}$ with $\alpha,\beta >0$, we call spherical $L$-codes \emph{uniacute spherical codes}. Here the set of allowed angles contains exactly one acute angle.
Bukh's result prompts the determination of the growth rate of $N_L(d)$.

\begin{problem}[Rate of uniacute spherical codes] \label{prob:rate}
Given $L \subseteq [-1, -\beta] \cup \{\alpha\}$ with $\alpha, \beta \in (0,1)$, determine
$\lim_{d\to\infty} N_L(d)/d$.
\end{problem}

Bukh showed that $\limsup_{d\to\infty} N_L(d)/d \le 2^{O(\beta^{-2})}$.
Balla, Dr\"{a}xler,  Keevash, and Sudakov \cite{BDKS18} gave the following substantial improvement.
We call it the ``uniform bound.''

\begin{theorem}[Uniform bound~\cite{BDKS18}] \label{thm:uniform-lim}
Let $L = [-1, -\beta] \cup \{\alpha\}$ with $\alpha,\beta \in (0,1)$. Let
$p = \floor{\alpha/\beta} + 1$.
Then 
\[
\limsup_{d \to \infty} \frac{N_L(d)}{d} \le 2p.
\]
\end{theorem}

The parameter $p$ plays an important role in the problem, as we will see later in the global structure theorem. Also, as we will show, for each $p$, there is some $(\alpha,\beta)$ with $p = \floor{\alpha/\beta} + 1$ so that the above uniform bound is tight. 

A notable result from \cite{BDKS18} is an analogous uniform bound for equiangular lines:
\[
\limsup_{d \to \infty} \frac{N_{\set{\pm\alpha}}(d)}{d} \le 2,
\]
and equality occurs if and only if $\alpha = 1/3$.
The precise limit was later determined for all $\alpha$ in \cite{JTYZZ21}.

\cref{thm:uniform-lim} leads to the following natural question.

\begin{question}
    For which $(\alpha,\beta)$ is the uniform bound in \cref{thm:uniform-lim} tight?
\end{question}

We give a complete answer to this question. The main result is stated below.
More refined bounds will be given later in \cref{thm:detail}.

\begin{theorem} \label{thm:main}
Let $L = [-1, -\beta] \cup \{\alpha\}$ with $\alpha,\beta \in (0,1)$. 
Let 
\[
p = \floor{\frac{\alpha}{\beta}} + 1
\qquad 
\text{and} \qquad
\lambda = \frac{1-\alpha}{\alpha +\beta}.
\]
If any of the following holds:
\begin{enumerate}[(a)]
    \item $\alpha = \beta$ and $\lambda^2 \in \ZZ$,
    \item $\alpha/\beta \in \ZZ$, $\alpha \ge 2\beta$, and $\lambda \in \ZZ$,
    \item $\alpha/\beta \notin \ZZ$ and $\lambda \ge 1$,
\end{enumerate}
then 
\[
    N_L(d) = 2pd + O_L(1),
\]
and otherwise, there is some absolute constant $c > 0$ such that
\[
\limsup_{d \to \infty} \frac{N_L(d)}{d} \le 2p-c.
\]
\end{theorem}

\subsection{Methods}
Our work builds on methods for equiangular lines, where the main result is recalled below.

\begin{theorem}[Equiangular lines with a fixed angle~\cite{JTYZZ21}] \label{thm:equiangular}
    Let $ \alpha \in (0,1)$ and $\lambda = (1-\alpha)/\alpha$. Let $k = k(\lambda)$ denote the smallest $k$ such that there exists a $k$-vertex graph whose adjacency matrix has top eigenvalue exactly $\lambda$; set $k = \infty$ is none exists. 
    \begin{enumerate}[(a)]
        \item \label{itm:equiangular-a} If $k < \infty$, then $N_{\set{\pm \alpha}} (d) = \floor{k(d-1)/(k-1)}$ for all $d > d_0(\alpha)$ for some $d_0(\alpha)$;
        \item \label{itm:equiangular-b} If $k = \infty$, then $N_{\set{\pm\alpha}}(d) = d + o(d)$.
    \end{enumerate}
\end{theorem}

A natural generalization of equiangular lines is spherical two-distance sets, namely spherical $L$-codes for $L = \set{-\beta,\alpha}$ with $\alpha,\beta > 0$. 
The limiting value of $N_L(d)/d$ was conjectured in \cite{JTYZZ23} and proved whenever $\alpha < 2\beta$ or $(1-\alpha)/(\alpha+\beta) < \gamma^{1/2} + \gamma^{-1/2} \approx 2.02$ where $\gamma$ is the unique real solution to $\gamma^3 = \gamma + 1$ \cite{JP21+,JTYZZ23}.
In this paper, we propose an extension of the conjectured limit from spherical two-distance sets to more general uniacute codes. To motivate and state the conjecture, we need to first discuss the solution framework.

Given a spherical code $C$, define its \emph{Gram graph} $G_C$ to be the edge-weighted complete graph on vertex set $C$ and such that edge $uv$ has weight $\ang{u,v}$.
Let $G_C^-$ be the \emph{negative graph}, which is the simple graph on consisting of negative edges of $G_C$ (and then forgetting the edge weights).
Likewise let $G_C^+$ be the \emph{positive graph}, a simple graph consisting of positive edges of $G_C$.

First, we show that $G_C$ is close to a $p$-block ``template,'' a notion that we will make precise later on. See \cref{fig:template} for an illustration of the template.
This idea originates from \cite{BDKS18} and was further developed in \cite{JTYZZ21,JTYZZ23}. 
A refined statement is given below and proved in \cref{sec:global}.

\begin{figure}[h]
\definecolor{wwccff}{rgb}{0.4,0.8,1}
\begin{tikzpicture}[scale = 0.4]
\fill[fill=wwccff!30] (-3.9993347497973892,-0.07294901687515672) -- (-1.9364916731037085,3.5) -- (0,2) -- (-1.7320508075688779,-1) -- cycle; 
\node at (-2, 1) {\large$\gamma_{12}$};

\fill[fill=wwccff!30] (-1.7320508075688779,-1) -- (-2.0628430766936816,-3.427050983124842) -- (2.06284307669368,-3.427050983124843) -- (1.7320508075688772,-1) -- cycle;
\node at (0, -2.25) {\large$\gamma_{13}$};

\fill[fill=wwccff!30] (1.7320508075688772,-1) -- (3.9993347497973892,-0.07294901687515976) -- (1.9364916731037083,3.5) -- (0,2) -- cycle;
\node at (2, 1) {\large$\gamma_{23}$};

\draw [draw=none, fill=wwccff] (0, 4) circle (2cm) node {\large$\alpha$};
\node[above] at (0, 6) {$V_2$};
\draw [draw=none, fill=wwccff] (-3.464101615137755, -2) circle (2cm) node {\large$\alpha$};
\node[left] at (-5.4641016151377535, -2) {$V_1$};
\draw [draw=none, fill=wwccff] (3.4641016151377535, -2) circle (2cm) node {\large$\alpha$};
\node[right] at (5.4641016151377535, -2) {$V_3$};
\end{tikzpicture}

\caption{An example of a template with $p$ vertex parts $V_1, \dots, V_p$. Edges within each $V_i$ have weight $\alpha$. Edges between $V_i$ and $V_j$ have weight $\gamma_{ij} \in L \setminus \{\alpha\}$ for each $i \ne j$.} \label{fig:template}
\end{figure}
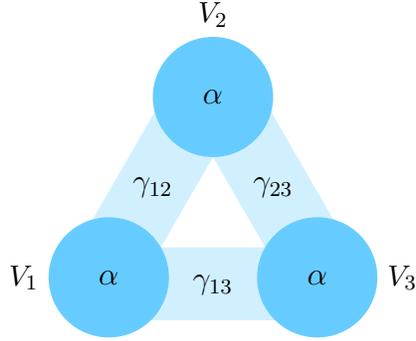

\begin{definition}
    An edge-weighted graph $H$ is a \textit{$\Delta$-modification} of another edge-weighted graph $G$ on the same vertex set if one can obtain $H$ from $G$ by changing edges and weights on a graph of maximum degree at most $\Delta$.
\end{definition}

\begin{theorem}[Global structure theorem]\label{thm:structure}
Fix $ \alpha, \beta \in (0,1)$ and let $L \subseteq [-1,-\beta]\cup\{\alpha\}$ and $p = \lfloor \alpha/ \beta \rfloor + 1$. There exist constants $\Delta_0$ and $K$ (both depending only on $\alpha$ and $\beta$) such for any spherical $L$-code $C$ and any integer $\Delta \geq \Delta_0$, we can remove at most $\Delta^2$ vectors from $C$ so that the Gram graph of remaining vectors is a $\Delta$-modification of a complete edge-weighted graph of the following form:
\begin{enumerate}[(a)]
	\item the vertex set can be partitioned as $V_1 \sqcup \cdots \sqcup V_p$ (some parts could be empty);
	\item for $1 \leq i \leq p$, all edges in $V_i$ have weight $\alpha$;
	\item for $1 \leq i  < j \leq p$, there is some $\gamma_{ij} \in L \setminus \{\alpha\}$ such that every edge between $V_i$ and $V_j$ has weight within $K \Delta^{-1/2}$ of $\gamma_{ij}$. \end{enumerate}
\end{theorem}

Informally, the global structure theorem says that after removing $O_L(1)$ vertices, and changing edge-weights on a subgraph of maximum degree $O_L(1)$, the resulting edge weights are very close to a $p$-block structure. \cref{thm:structure} is proved in \cref{sec:global}.

\begin{remark}[Necessity of vertex removal]
In \cref{thm:structure}, the removal of $O_L(1)$ vectors is necessary.
This is shown in the proof of \cref{thm:detail}\ref{itm:detail-non}.
In contrast, for equiangular lines, this step is unnecessary. (See \cite[Theorem 2.1]{JTYZZ21}; also see \cite{Bal21+} for quantitative improvements specifically in the context of equiangular lines. It should be noted that the improvements in \cite{Bal21+} do not seem to apply to spherical two-distance sets or more general uniacute spherical codes.)
\end{remark}

After applying the global structure theorem, the heart of the problem then lies in analyzing the bounded degree subgraph where weights were modified. 
For equiangular lines, to prove \cref{thm:equiangular}, we use that a connected bounded degree graph has sublinear eigenvalue multiplicity~\cite{JTYZZ21}.
In this paper, we will use a slight extension of the eigenvalue multiplicity result to edge-weighted graphs. The version we need is stated below. We omit the proof as it is identical to the proof of \cite[Theorem 2.1]{JTYZZ21}.

Given an $n$-vertex edge-weighted graph $G$, we define its weighted adjacency matrix $A_G$ as a matrix whose $(i,j)$-entry is the edge-weight on edge $ij$, and zero if $ij$ is not an edge. Label its eigenvalues by $\lambda_1(A_G) \ge \cdots \ge \lambda_n(A_G)$, and we refer to $\lambda_j(A_G)$ as the \emph{$j$th eigenvalue of $G$}.

\begin{theorem}[Sublinear eigenvalue multiplicity~\cite{JTYZZ21}]\label{thm:eigen-mult} 
Given a connected edge-weighted graph on $n$ vertices with maximum degree $\Delta$ and whose edge weights lie in the interval $[1, 1+\nu]$, its $j$th eigenvalue has multiplicity at most $O_{\Delta,j,\nu}(n/\log\log n)$.
\end{theorem}

\begin{remark}[Quantitatve bounds and error terms]
The $n/\log\log n$ bound in \cref{thm:eigen-mult} is the main source of the bottleneck in the quantitative bounds in \cref{thm:equiangular}. Specifically, the proof (together with later refinements by Balla~\cite{Bal21+}) gives $d_0(\alpha) = \exp(\alpha^{-Ck})$ in \cref{thm:equiangular}\ref{itm:equiangular-a} and $O(\log(1/\alpha) d/\log\log d)$ as the $o(d)$ error term in \cref{thm:equiangular}\ref{itm:equiangular-b}. See \cite{HSSZ22} lower bound constructions of connected bounded degree graphs with second eigenvalue multiplicity on the order of $\sqrt{n/\log n}$). Also, for \cref{thm:equiangular}\ref{itm:equiangular-b}, Schildkraut~\cite{Sch23+} exhibited infinitely many $\alpha$'s where the $o(d)$ error term cannot be replaced by $o(\log\log d)$.
\end{remark}

Prior to the discovery of \cref{thm:eigen-mult}, partial progress on the equiangular lines problem relied on ad hoc spectral graph theory arguments \cite{BDKS18,JP20}. Despite the success in addressing the equiangular lines problem, a significant challenge arises when attempting to extend the solution to other uniacute spherical code problems. 
A missing ingredient appears to be the lack of a useful generalization of \cref{thm:eigen-mult} to the types of graphs that arise in \cref{thm:structure}.
Indeed, for spherical two-distance sets \cite{JTYZZ23,JP21+}, various workarounds have been used to address certain cases of the problem.
The papers \cite{JP20,JP21+} solved certain forbidden subgraph characterization problems in graphs and signed graphs, which led to partial solutions of equiangular lines and spherical two-distance sets problems.

To show our main result, \cref{thm:main}, after applying the global structure theorem to decompose the Gram graph into $p$ parts, we deduce that unless the modification consists of a near perfect matching inside each of the $p$ parts, the size of the code must be much less than $2pd$ (\cref{sec:rank-bound}).
Once we deduce the matching structure within each part, we then analyze the modifications between parts (\cref{sec:upper}).

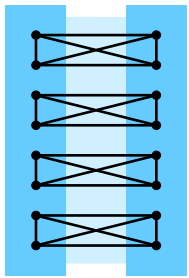
\begin{figure}[h]
\definecolor{wwccff}{rgb}{0.4,0.8,1}
\begin{tikzpicture}[scale = 0.8]
\fill[line width=0pt,color=wwccff,fill=wwccff!30] (-2.005278955937325,1.8055364510038712) -- (1,1.8) -- (1,-2.3) -- (-2,-2.3) -- cycle;

\fill[line width=1pt,color=wwccff,fill=wwccff] (-1,2) -- (-2,2) -- (-2,-2.5) -- (-1,-2.5) -- cycle;
\fill[line width=1pt,color=wwccff,fill=wwccff] (1,2) -- (0,2) -- (0,-2.5) -- (1,-2.5) -- cycle;

\draw [line width=1pt] (-1.5,1.5)-- (0.5,1);
\draw [line width=1pt] (0.5,1.5)-- (-1.5,1);
\draw [line width=1pt] (-1.5,1)-- (-1.5,1.5);
\draw [line width=1pt] (-1.5,1.5)-- (0.5,1.5);
\draw [line width=1pt] (0.5,1.5)-- (0.5,1);
\draw [line width=1pt] (0.5,1)-- (-1.5,1);
\draw [line width=1pt] (-1.5,0.5)-- (0.5,0.5);
\draw [line width=1pt] (0.5,0)-- (0.5,0.5);
\draw [line width=1pt] (-1.5,0)-- (-1.5,0.5);
\draw [line width=1pt] (-1.5,0)-- (0.5,0);
\draw [line width=1pt] (0.5,0)-- (-1.5,0.5);
\draw [line width=1pt] (-1.5,0)-- (0.5,0.5);
\draw [line width=1pt] (-1.5,-0.5)-- (0.5,-1);
\draw [line width=1pt] (0.5,-1)-- (0.5,-0.5);
\draw [line width=1pt] (0.5,-0.5)-- (-1.5,-1);
\draw [line width=1pt] (-1.5,-1)-- (-1.5,-0.5);
\draw [line width=1pt] (-1.5,-0.5)-- (0.5,-0.5);
\draw [line width=1pt] (0.5,-1)-- (-1.5,-1);
\draw [line width=1pt] (-1.5,-1.5)-- (0.5,-2);
\draw [line width=1pt] (0.5,-1.5)-- (-1.5,-2);
\draw [line width=1pt] (-1.5,-2)-- (-1.5,-1.5);
\draw [line width=1pt] (-1.5,-1.5)-- (0.5,-1.5);
\draw [line width=1pt] (0.5,-1.5)-- (0.5,-2);
\draw [line width=1pt] (0.5,-2)-- (-1.5,-2);

\draw [fill=black] (-1.5,1.5) circle (2pt);
\draw [fill=black] (-1.5,1) circle (2pt);
\draw [fill=black] (0.5,1.5) circle (2pt);
\draw [fill=black] (0.5,1) circle (2pt);
\draw [fill=black] (-1.5,0.5) circle (2pt);
\draw [fill=black] (-1.5,0) circle (2pt);
\draw [fill=black] (0.5,0.5) circle (2pt);
\draw [fill=black] (0.5,0) circle (2pt);
\draw [fill=black] (-1.5,-0.5) circle (2pt);
\draw [fill=black] (-1.5,-1) circle (2pt);
\draw [fill=black] (0.5,-0.5) circle (2pt);
\draw [fill=black] (0.5,-1) circle (2pt);
\draw [fill=black] (-1.5,-1.5) circle (2pt);
\draw [fill=black] (-1.5,-2) circle (2pt);
\draw [fill=black] (0.5,-1.5) circle (2pt);
\draw [fill=black] (0.5,-2) circle (2pt);
\end{tikzpicture}

\caption{An illustration of a modular code obtained by modifying a template on disjoint components.} \label{fig:modular}
\end{figure}

\subsection{Modularity conjecture}

One method for constructing uniacute spherical codes is to begin with the $p$-block template from the global structure theorem and then incorporate modifications that are ``modular.'' We introduce the concept of \emph{modular spherical $L$-codes}, or simply \emph{modular codes}, to describe this type of construction. 
These codes are modular in the sense that larger modular codes can be constructed by repeating small modular codes.
See \cref{fig:modular} for an illustration.

In all solved cases of the rate problem for uniacute spherical codes, such as for equiangular lines, the optimal rate can be achieved using modular codes. We hypothesize that a similar phenomenon holds for general uniacute codes. 
To formalize these notions, we formulate a \emph{modularity conjecture} in  \cref{sec:modularity}.
The modularity perspective was already helpful here in our work proving \cref{thm:main}.

\subsection{Detailed results}

We now state refined versions of \cref{thm:uniform-lim,thm:main}.

\begin{theorem}\label{thm:general-upper}
	Fix $\alpha, \beta \in (0,1)$. Let $p = \floor{\alpha/\beta} + 1$ and $\lambda = (1-\alpha)/(\alpha + \beta)$. Set $L = [-1,-\beta] \cup \{\alpha\}$. Then the maximum cardinality of an $L$-code in $\RR^d$ satisfies
\[
    N_L(d) \leq \begin{cases}
        2pd + O_L(1) & \text{ if } \lambda \geq 1,\\
        pd + o(d) & \text{ if } \lambda < 1.\\
    \end{cases}
\]

\end{theorem}

\begin{theorem} \label{thm:detail}
Let $L = [-1, -\beta] \cup \{\alpha\}$ with $\alpha,\beta \in (0,1)$. 
Let 
\[
p = \floor{\frac{\alpha}{\beta}} + 1
\qquad 
\text{and} \qquad
\lambda = \frac{1-\alpha}{\alpha +\beta}.
\]
\begin{enumerate}[(a)]
    \item \label{itm:detail-1}
    Suppose $\alpha/\beta = 1$. 
        \begin{enumerate}[(i)]
            \item \label{itm:detail-1int} If $\lambda^2 \in \ZZ$, then there exist positive integers $d_0 = d_0(\alpha)$ and $m = m(\alpha)$ such that for all $d \ge d_0$,
            \[
            N_L(d) \le 4(d-1),
            \]
            and, whenever $d-1$ is divisible by $m$,
            \[
            N_L(d) \ge 4(d-1).
            \]
            \item \label{itm:detail-1non} If $\lambda^2 \notin \ZZ$, then 
            \[N_L(d) \le 11d/3 + O_L(1).\]
        \end{enumerate}
    \item \label{itm:detail-2} Suppose $\alpha/\beta \in \ZZ$ and $\alpha/\beta \geq 2$.
        \begin{enumerate}[(i)]
            \item \label{itm:detail-2int}
                If $\lambda \in \ZZ$, then there exists a positive integer $m = m(\alpha, \beta)$ such that
                \[
                N_L(d) \le 2pd + O_L(1),
                \]
                and, whenever $d-p+1$ is divisible by $m$,
                \[
                N_L(d) \ge 2p(d-p+1).
                \]
            \item \label{itm:detail-2non} If $\lambda \notin \ZZ$, then 
            \[
            N_L(d) \le \left(2p - \frac{1}{2}\right)d + O_L(1).
            \]
        \end{enumerate}
    \item  \label{itm:detail-non} 
    Suppose $\alpha / \beta \notin \ZZ$. Then
    \[
        N_L(d) = \begin{cases}
                2pd + O_L(1) & \text{ if } \lambda \geq 1,\\
                pd + o(d) & \text{ if } \lambda < 1. \\
                 \end{cases}
    \]
    \item \label{itm:error-term} For all integers $p$ and $K$ there exists $\varepsilon$ such that whenever $p-\varepsilon < \alpha/\beta < p$, there exists $d = d_0(\alpha,\beta)$ such that for all $d \ge d_0$,
            \[
            N_L(d) \ge 2pd + K.
            \]
\end{enumerate}
\end{theorem}

The proofs are given as follows.

Upper bounds:
\begin{enumerate}
\item[\ref{itm:detail-1}]
\ref{itm:detail-1int} \cref{prop:upper} \quad 
\ref{itm:detail-1non} \cref{prop:algebraic-upper}\ref{itm:algebraic-upper-2}
\item [\ref{itm:detail-2}]
\ref{itm:detail-2int} \cref{thm:general-upper} \quad 
\ref{itm:detail-2non} \cref{prop:algebraic-upper}\ref{itm:algebraic-upper-p}
\item [\ref{itm:detail-non}]  \cref{thm:general-upper}
\end{enumerate}

Lower bounds:
\begin{enumerate}
\item[\ref{itm:detail-1}]
\ref{itm:detail-1int} \cref{prop:lower}\ref{itm:lower-2p}\ref{itm:lower-4}

\item [\ref{itm:detail-2}]
\ref{itm:detail-2int} \cref{prop:lower}\ref{itm:lower-2p}\ref{itm:lower-int-2p}

\item [\ref{itm:detail-non}]  \cref{prop:lower}\ref{itm:lower-nint-p} and \cref{prop:lower}\ref{itm:lower-2p}\ref{itm:lower-nint-2p}
\item [\ref{itm:error-term}] \cref{prop:lower-extra}.
\end{enumerate}

\begin{remark}
In \ref{itm:error-term}, we see that the $O_L(1)$ term in the uniform bound $N_L(d) \le 2pd + O_L(1)$ needs to be large when $\alpha/\beta$ falls in the upper end of its allowed interval $[p-1,p)$. 
The construction showing \ref{itm:error-term} also shows that in the global structure theorem, \cref{thm:structure}, it is sometimes necessary to remove a large number of vertices before partitioning the vertex set into $p$ parts.
\end{remark}

\subsection{Further directions.}
This paper makes progress on the rate problem for uniacute spherical codes (\cref{prob:rate}) by characterizing all $L = [-1,-\beta]\cup\set{\alpha}$ where the uniform bound \cref{thm:uniform-lim} is tight. 
It remains open to determine the rate in other cases.

Our proposed modularity conjecture, discussed in \cref{sec:modularity}, provides a characterization of the rate.
Ideas from \cite{BDKS18,JP20,JP21+,JTYZZ23} may be helpful in establishing more cases, although a full solution may hinge on a proper generalization of \cref{thm:eigen-mult}.

Balla, Dr\"{a}xler, Keevash, and Sudakov \cite{BDKS18} showed that for $L = [-1, -\beta] \cup \{\alpha_1, \cdots, \alpha_k\}$ with $0 < \beta \leq 1$ and $0 \leq \alpha_1 < \cdots < \alpha_k < 1$,
\[
\limsup_{d \to \infty} \frac{N_L(d)}{d^k} \leq 2^k(k-1)!\paren{1 + \floor{\frac{\alpha_1}{\beta}}}
\]
and, for each $k$, there exists $L$ with a positive limsup.
There is much more work to do to better understand these ``$k$-acute spherical codes'' or ``multiacute codes.''

Another related direction concerns replacing real unit vectors by complex unit vectors. This is related to complex equiangular lines. 
Balla~\cite{Bal21+} recently showed that the maximum number of complex unit vectors in $\CC^d$ whose pairwise inner product has magnitude $\alpha$ is $O_\alpha(d)$. 
A related problem concerns configurations of subspaces with prescribed angles~\cite{LS73subspace,BS19}.

\section{Modularity}\label{sec:modularity}

To motivate the discussion in this section, we first consider two example constructions. Instead of dealing with spherical codes explicitly, \cref{lem:gram} will allow us to work with Gram matrices instead. 

\begin{lemma}\label{lem:gram}
	Let $L \subseteq [-1,1)$. There exists a spherical $L$-code of size $n$ in $\RR^d$ if and only if there exists an $n \times n$ positive semidefinite matrix of rank at most $d$ with all diagonal entries being $1$ and all non-diagonal entries belonging to $L$.
\end{lemma}

\begin{proof}
	Given a spherical $L$-code of size $n$ in $\RR^d$, its Gram matrix is an $n\times n$ positive semidefinite matrix of rank at most $d$ with all diagonal entries being $1$ and  all non-diagonal entries being in $L$. On the other hand, given such a matrix, Cholesky decomposition recovers the spherical $L$-code.
\end{proof}

\begin{example}\label{ex:equiangular-construction}

Lemmens and Seidel~\cite{LS73} showed that $N_{\{\pm 1/3\}}(d) = 2(d-1)$ for all sufficiently large dimensions $d$. Our first example will be a tight lower bound construction for $N_{\{\pm 1/3\}}(d)$ for large $d$. 

	Let $L = \{\pm 1/3\}$. Observe that by \cref{lem:gram}, there is a spherical $L$-code of size two in $\RR^2$ with the Gram matrix 
\[
	\begin{pmatrix}
		1 &  -1/3 \\
		-1/3 & 1 \\
	\end{pmatrix} = \frac{1}{3}\begin{pmatrix}
						 1 &  1 \\
						1 & 1 
						\end{pmatrix} + \frac{2}{3}\begin{pmatrix}
													 \phantom{+}1 &  -1 \\
													-1 & \phantom{+}1 
													\end{pmatrix}.
\]
Note also that the matrix
\[
	P = \begin{pmatrix}
		 \phantom{+}1 &  -1 \\
		-1 & \phantom{+}1 
		\end{pmatrix}
\]
is positive semidefinite matrix of rank $1$. For $d \geq 2$ we can chain together $d-1$ copies of $P$ to form positive semidefinite matrix $Q$ of rank $(d-1)$ with block form  
\[
	Q = \frac{2}{3}\begin{pmatrix}
		  P &  &  &  \\
		   & P &  &  \\
		   &  & \ddots &  \\
		   &  &  & P \\
		    	\end{pmatrix}.
\]
Let $J$ be the $2(d-1) \times 2(d-1)$ all-ones matrix and 
\[
	T = \frac{1}{3} J  \quad \text{ and } \quad M = T + Q.
\]
Since $M$ is sum of positive semidefinite matrices, it is positive semidefinite. Moreover, by sub-additivity of rank, $M$ is a $2(d-1) \times 2(d-1)$  matrix of rank at most $d$. Entries of $M$ are
\[
M = \begin{pmatrix}
    \begin{array}{cccc}
       \begin{matrix}
 		 1 & -1/3 \\
  		-1/3 & 1
  \end{matrix} & \begin{matrix}
  					1/3 & \phantom{-}1/3 \\
  					1/3 & \phantom{-}1/3
  				\end{matrix} & \cdots & \begin{matrix}
 										 1/3 & \phantom{-}1/3 \\
  										1/3 & \phantom{-}1/3
 										 \end{matrix} \\
  \begin{matrix}
  1/3 & \phantom{-}1/3 \\
  1/3 & \phantom{-}1/3
  \end{matrix} & \begin{matrix}
 				 1 & -1/3 \\
 				 -1/3 & 1
 				 \end{matrix} & \cdots & \begin{matrix}
  										1/3 & \phantom{-}1/3 \\
  										1/3 & \phantom{-}1/3
  										\end{matrix} \\
  \vdots & \vdots & \ddots & \vdots \\
 
   \begin{matrix}
  1/3 & \phantom{-}1/3 \\
  1/3 & \phantom{-}1/3
  \end{matrix} & \begin{matrix}
 				 1/3 & \phantom{-}1/3 \\
 				 1/3 & \phantom{-}1/3
 				 \end{matrix} & \cdots & \begin{matrix}
  										1 & -1/3 \\
  										-1/3 & 1
  										\end{matrix} \\		
    \end{array}
\end{pmatrix}.
\]
The diagonal entries of $M$ are all $1$ and non-diagonal entries all lie in $\{\pm 1/3\}$. Hence, by \cref{lem:gram} there is a spherical $\{\pm1/3\}$-code of size $2(d-1)$ in $\RR^d$. 

\medskip

Observe that nearly all entries of $M$ coincide with those of the low rank matrix $T$. In this sense, one can think of $T$ as a template. Observe also that the difference of the Gram matrix and the template $M - T = Q$ is a positive semidefinite matrix. 
\end{example}

\begin{example}\label{ex:uniacute-construction}

\cref{thm:detail} implies that the maximum cardinality of a spherical $[-1, -1/3] \cup \{1/3\}$-code is exactly $4(d - 1)$ for all sufficiently large dimensions $d$. Our second example will be the sharp construction of spherical $[-1, -1/3] \cup \{1/3\}$-codes for large $d$.

	By \cref{lem:gram}, there is a spherical $[-1, -1/3] \cup \{1/3\}$-code of size 4 in $\RR^2$ with Gram matrix
	\[
		\begin{pmatrix}
		1 & -1/3 & -1 & 1/3 \\
		-1/3 & 1 & 1/3 & -1 \\
		-1 & 1/3 & 1 & -1/3 \\
		1/3 & -1 & -1/3 & 1
	\end{pmatrix}
		=
		\frac{1}{3}
			\begin{pmatrix}
			\phantom{+}1 & \phantom{+}1 & -1 & -1 \\
			\phantom{+}1 & \phantom{+}1 & -1 & -1 \\
			-1 & -1 & \phantom{+}1 & \phantom{+}1 \\
			-1 & -1 & \phantom{+}1 & \phantom{+}1
			\end{pmatrix}
		+
		\frac{2}{3}
			\begin{pmatrix}
			\phantom{+}1 & -1 & -1 & \phantom{+}1 \\
			-1 & \phantom{+}1 & \phantom{+}1 & -1 \\
			-1 & \phantom{+}1 & \phantom{+}1 & -1 \\
			\phantom{+}1 & -1 & -1 & \phantom{+}1
		\end{pmatrix}.
\]
Note also that the matrix 
\[
\begin{pmatrix}
	P_{11} & P_{12} \\
	P_{21} & P_{22}
\end{pmatrix}
		=
		\begin{pmatrix}
		\begin{array}{cc|cc}
			\phantom{+}1 & -1 & -1 & \phantom{+}1 \\
			-1 & \phantom{+}1 & \phantom{+}1 & -1 \\
			\hline
			-1 & \phantom{+}1 & \phantom{+}1 & -1 \\
			\phantom{+}1 & -1 & -1 & \phantom{+}1
		\end{array}
\end{pmatrix}
\]
is positive semidefinite matrix of rank $1$. For $d \geq 2$ we can chain together $(d-1)$ copies of the matrix to form positive semidefinite matrix $Q$ of rank $(d-1)$
\[
	Q = \frac{2}{3}\begin{pmatrix}
	\begin{array}{c|c}
		\begin{array}{ccc}
			P_{11} &  &    \\
			 &  \ddots &  \\
			 &   & P_{11}
		\end{array}
		&
		\begin{array}{ccc}
			P_{12} &   &  \\
			 &   \ddots &  \\
			 &    & P_{12}
		\end{array}
		\\
		\hline
		\begin{array}{ccc}
			P_{21} &    &  \\
			 &   \ddots &  \\
			 &   & P_{21}
		\end{array}
		&
		\begin{array}{ccc}
			P_{22} &  &    \\
			 &   \ddots &  \\
			 &    & P_{22}
		\end{array}
	\end{array}
\end{pmatrix}.
\]
Let  $J$ be $2(d-1)\times 2(d-1)$ all-ones matrix and define $T$ and $M$ by 
\[
	T = \frac{1}{3}\begin{pmatrix}
		\phantom{+}J & -J \\
		-J & \phantom{+}J \\
	\end{pmatrix} \quad \text{ and } \quad M = T + Q.
\]
Since $M$ is a sum of positive semidefinite matrices, it is positive semidefinite. Moreover, by subadditivity of rank, $M$ is a $4(d-1) \times 4(d-1)$ matrix of rank at most $d$. Finally, diagonal entries of $M$ are all $1$ and non-diagonal entries are in $\{-1, \pm 1/3\}$. Hence, by \cref{lem:gram}, there exists a spherical $[-1, -1/3]\cup \{1/3\}$-code of size $4(d-1)$ in $\RR^d$.

Observe also that the majority of entries of $M$ agree with those of the matrix $T$. Once again, one can think of the matrix $T$ as a template and note that the difference of the Gram matrix and the template $M -T = Q$ is a positive semidefinite matrix.
\end{example}

  Bearing these examples in mind, we will define modular codes.

\begin{definition}[Template]\label{def:template}
Let $\alpha, \beta \in (0,1)$ and $L \subseteq [-1,-\beta]\cup\{\alpha\}$ with $\alpha \in L$. We say that a square matrix $T$ is a \textit{template (with respect to $L$)} if it has a block form
	\[
		T = \begin{pmatrix}
    T_{11} & T_{12} & \dots & T_{1m} \\
    T_{21} & T_{22} & \dots & T_{2m} \\
    \vdots & \vdots & \ddots & \vdots \\
    T_{m1} & T_{m 2} & \dots & T_{mm}
		\end{pmatrix}
	\]
	satisfying the following properties:
	\begin{enumerate}[label=(\roman*)]
		\item $T$ is positive semidefinite;
		\item For each $1 \leq i \leq m$, all entries of $T_{ii}$ are equal to $\alpha$;
		\item For each $1 \leq i < j \leq m$, all entries of $T_{ij}$ are equal to $\gamma_{ij}$ for some $\gamma_{ij} \in L \setminus \{\alpha\}$.
	\end{enumerate}
\end{definition}

\begin{definition}[Modular code]\label{def:modular}
Let $\alpha, \beta \in (0,1)$ and $L \subseteq [-1,-\beta]\cup\{\alpha\}$ with $\alpha \in L$.
	We say that an $L$-code $C$ is \textit{modular} if there exists a template $T$ with respect to $L$ such that $M_C - T$ is positive semidefinite, where $M_C$ is the Gram matrix of possibly reordered $C$. 

 We denote $N_L^{mod}(d)$ to be the maximum size of a modular $L$-code in $\RR^d$.
\end{definition}

\cref{ex:equiangular-construction} is a construction of a modular $\{\pm 1/3\}$-code and \cref{ex:uniacute-construction} is a construction of a modular $[-1,-1/3]\cup\{1/3\}$-code. In fact, all uniacute spherical codes in high dimensions that are known to be tight arise from modular ones. Motivated by this, we form the following conjecture.

\begin{conjecture}[Modularity]\label{conj:modularity}
Let $\alpha,\beta \in (0,1)$ and $L \subseteq [-1,-\beta]\cup\{\alpha\}$ with $\alpha \in L$. Then
\[
		N_L(d) = N_L^{mod}(d) + o(d) \text{ as } d \to \infty.
\]
\end{conjecture}

Our modularity conjecture is known to hold in many special cases, namely for equiangular lines $L = \set{\pm \alpha}$ \cite{JTYZZ21}, certain cases of the spherical two-distance set problem $L = \set{-\beta, \alpha}$ \cite{JTYZZ23, JP21+}, as well as for $L = [-1,-\beta] \cup \set{\alpha}$ when the uniform bound, \cref{thm:uniform-lim}, is tight, as we prove in this paper.

For equiangular lines, \cref{thm:equiangular}, in case \ref{itm:equiangular-a} where $k < \infty$, we have a stronger conclusion $N_L(d) = N_L^{mod}(d)$ for all sufficiently large $d$.
On the other hand, in case \ref{itm:equiangular-b} where $k = \infty$, the $N_L(d) - N_L^{mod}(d) = o(d)$ error term is $O(d/\log\log d)$, and Schildkraut \cite{Sch23+} proved that there exist infinitely many $\alpha$'s for which $N_{\{\pm \alpha\}}^{mod}(d) = d-1$ and $N_{\{\pm \alpha\}}(d) - N_{\{\pm \alpha\}}^{mod}(d) \ge  \Omega_\alpha(\log \log d)$. As a result, we cannot replace the $o(d)$ term in the modularity conjecture by $O_L(1)$.

In this paper, we prove the modularity conjecture for $L = [-1,-\beta]\cup\set{\alpha}$ in cases where the uniform bound (\cref{thm:uniform-lim}) is tight. In some instances, such as \cref{thm:detail}\ref{itm:detail-1}\ref{itm:detail-1int}, the error term is zero for infinitely many $d$. In other cases of \cref{thm:detail}, the error term is $O_L(1)$, but this error term can be arbitrarily large as a function of $L$, as seen in \cref{thm:detail}\ref{itm:error-term}.

\section{Lower Bound Constructions}\label{sec:lower}

In this section we will prove the lower bounds of \cref{thm:detail}.

\begin{proposition}\label{prop:lower}
    Fix $\alpha, \beta \in (0,1)$. Let $\lambda = (1-\alpha)/(\alpha + \beta)$ and $p = \lfloor \alpha/\beta \rfloor + 1$. Let $L = [-1,-\beta]\cup \{\alpha\}$. Then, there exist $m = m(\alpha, \beta)$ and $d_0 = d_0(\alpha, \beta)$ such that
    \begin{enumerate}[(a)]
        \item \label{itm:lower-nint-p} $N_L(d) \geq p(d-p+1)$ whenever $\alpha/\beta \notin \ZZ$ and $d \geq d_0$; 
        \item \label{itm:lower-2p}  $N_L(d) \geq 2p(d-p+1)$ if any of the following hold: 
        \begin{enumerate}[(i)]
            \item \label{itm:lower-nint-2p} $\alpha/\beta \notin \ZZ$ and $\lambda \geq 1$ and $d \geq d_0$,
            \item \label{itm:lower-int-2p}  $\alpha/\beta \in \ZZ$ and $\alpha/\beta \geq 2$ and $\lambda \in \ZZ$ and $d-p+1$ is divisible by $m$,
            \item \label{itm:lower-4} $\alpha/\beta = 1$ and $\lambda^2 \in \ZZ$ and $d-p+1$ is divisible by $m$.
        \end{enumerate}
    \end{enumerate}
\end{proposition}

Our constructions rely on the existence of certain orthogonal matrices.

\begin{proposition}\label{prop:matrix}
    For any positive integer $k$, there exists a square matrix $Q$ with $\{0, \pm 1\}$ entries such that $Q^\intercal Q = k I$. 
\end{proposition}

\begin{proof} 
		Let $k$ be a positive integer. It suffices to construct a matrix $Q=Q_k$ with orthogonal columns such that each column has entries in $\{ 0, \pm 1\}$ and exactly $k$ of them are nonzero. We will first give the construction when $k = p$ is a prime number. Then, we will assemble $Q_k$ for arbitrary integers $k$ using the matrices $Q_p$ for each prime factor $p$ of $k$.
		
 		For $k = 1$ and $k = 2$  we can take the following matrices:
			\[
				Q_ 1 = \begin{pmatrix}
						1 
						\end{pmatrix} \quad \text{ and } \quad
				Q_2 = \begin{pmatrix}
						\phantom{+}1 & \phantom{+}1\\
						-1 & \phantom{+}1
						\end{pmatrix}.
			\]
		
		When $p$ is an odd prime number, define $Q_p$ to be a $(p+1)\times (p+1)$ matrix with entries
			\[
				(Q_p)_{ij} = \begin{cases}
							\legendre{i-j}{p} & \text{ if } 1 \leq i, j \leq p, \\
							0 & \text{ if } i = j = p+1, \\
							1 & \text{ otherwise, } \\
							\end{cases}
			\]
		where $\legendre{a}{p}$ denotes the Legendre symbol of $a$ and $p$. Clearly, all entries of $Q_p$ lie in $\{ 0, \pm 1\}$. Moreover, each row and column of $Q_p$ has exactly $p$ nonzero entries. For each $j \in [p]$,
			\[
				\bm e_{j}^\intercal  (Q_p^\intercal  Q_p) \bm e_{p+1} = \sum_{i = 1}^p \legendre{i-j}{p} = 0.
			\]
		Moreover, for all  $j, k \in [p]$ with $j \neq k$,
			\begin{align*}
				\bm e_{j}^\intercal  (Q_p^\intercal  Q_p) \bm e_{k} &= 1 + \sum_{i = 1}^p \legendre{i - j}{p}\legendre{i-k}{p} \\
															&= 1  + \sum_{i = 1}^p \legendre{i}{p}\legendre{i + (j-k)}{p} \\
															&= 1  + \sum_{i = 1}^{p-1} \legendre{i^2}{p}\legendre{1 + (j-k)i^{-1}}{p}\\
															&= 1  + \sum_{i = 1}^{p-1} \legendre{1 + (j-k)i^{-1}}{p}.
			\end{align*}
		Observe that the mapping $i \mapsto 1 + (j-k)i^{-1}$ maps $\FF_p \setminus \{0\}$ to  $\FF_p \setminus \{1\}$. Therefore,
			\begin{align*}
				\bm e_{j}^\intercal  (Q_p^\intercal  Q_p) \bm e_{k} &= 1  + \sum_{i = 1}^{p-1} \legendre{1 + (j-k)i^{-1}}{p}\\
															&= 1 + \left[\sum_{i = 0}^{p-1} \legendre{i}{p}\right] - \legendre{1}{p} = 0.
			\end{align*}
		So $Q_p$ is a $\{0, \pm 1\}$-valued matrix with $Q_p^\intercal Q_p = pI$.

        Finally, we can extend the construction of $Q_k$ to all positive integers $k$ via the Kronecker product $Q_{mn} = Q_m \otimes Q_n$. Here if $A=(a_{ij})_{1 \leq i, j \leq r}$ and $B=(b_{ij})_{1 \leq i,j\leq s}$ then $A \otimes B$ is an $rs \times rs$ matrix whose rows and columns are each indexed by $[r] \times [s]$ and whose $(i,i'),(j,j')$ entry is $a_{ij}b_{i'j'}$.
\end{proof}

\begin{proposition}\label{prop:matrix-copies}
    For any positive integer $k$ and $p$, there exist square matrices $Q_1, \ldots, Q_p$ of the same size with $\{0, \pm 1\}$ entries, such that $Q_i^\intercal Q_i = k I$ for each $i$ and $Q_i^\intercal Q_j$ has $\{0, \pm 1\}$ entries whenever $i \neq j$.
\end{proposition}

The following lifting procedure is based on an idea of Wocjan and Beth \cite{WB05}. 

\begin{definition}[Net]\label{def:net}
A collection of $m^2 \times m$ matrices $V_1, \ldots, V_p$ with  $\{0,1\}$ entries is called a \emph{$(p,m)$-net} if
\[
    V_i^\intercal V_j = \begin{cases}
                            m I_m & \text{ if } i = j, \\
                            J_m   & \text{ if } i \neq j. \\
                         \end{cases}
\]
Here $I_m$ is $m\times m$ identity matrix and $J_m$ is $m\times m$ all ones matrix.
\end{definition}

\begin{proposition}\label{porp:net}
    For any positive integer $p$ there exists $m_0$ such that there exists a $(p,m)$-net for each $m \geq m_0$.
\end{proposition}
\begin{proof}
Recall that an $m \times m$ matrix $L$ is a Latin square if each row and each column of $L$ is a permutation of $[m]$. Latin squares $L$ and $L'$ are orthogonal if the set $\{{(L^{(1)}_{ij}, L^{(2)}_{ij})}\}_{1 \leq i, j \leq m}$ is exactly equal to $[m]^2$. 
Chowla, Erd\H{o}s and Strauss \cite{CES60} showed that the number of $m \times m$ mutually orthogonal Latin squares tends to infinity as $m \to \infty$.

Given $p$, for sufficiently large $m \ge m_0(p)$, there exist mutually orthogonal Latin squares $L^{(1)}$, \dots, $L^{(p-2)}$ of order $m$. 
Let $L^{(p-1)}$ and $L^{(p)}$ be $m \times m$ matrices with $L^{(p-1)}_{ij} = i$ and $L^{(p)}_{ij} = j$ for all $i,j \in [m]^2$. 
For all $i, j, k \in [m]$ and $r \in [p]$, let $V_r$ have a $(im + j, k)$ entry equal to $1$ if $L^{r}_{ij} = k$ and $0$ otherwise. 
Then $V_1, \ldots, V_p$ form a $(p, m)$-net. The claim then follows.
\end{proof}

We will now use nets to lift matrix from \cref{prop:matrix} to prove \cref{prop:matrix-copies}.

\begin{proof}[Proof of \cref{prop:matrix-copies}]
    By \cref{porp:net} there exists a $(p,m)$ net for all $m \geq m_0(p)$. On the other hand, by \cref{prop:matrix} there exists a square matrix $Q$ with $\{0, \pm 1\}$-entries such that $Q^\intercal Q = kI$. We can assume that $Q$ is $m \times m$ matrix for some $m \geq m_0(p)$, because we can replace $Q$ by the matrix
   \[
				\begin{pmatrix}
					Q &  &   \\
					 & \ddots & \\
					& & Q
					\end{pmatrix}.
    \]
    Let us fix a $(p,m)$-net $V_1, \ldots, V_p$. For each $1 \leq i \leq p$ define $Q_i = Q \uparrow V_i$ to be an $m^2 \times m^2$ matrix formed by embedding columns of $Q$ into non-zero entries of columns of $V_i$.  Concretely, for each $1 \leq j \leq m$ and $1 \leq r \leq m$, we form the $(jm + r)$-th column of $Q_i$ by embedding $r$-th column of $Q$ into $j$-th column of $V_i$ as follows: the column $V_ie_j$ has exactly $m$ ones, and we simply replace these entries by the entries of the column $Qe_r$ in order. This procedure results in a matrix $Q_i$ with $\{0, \pm 1\}$ entries. Since the columns of $V_i$ have disjoint supports, we have that $Q_i^\intercal Q_i = kI$. Moreover, since the support of a column from $V_i$ and a column from $V_j$ have exactly one index in common, the entries of $Q_i^\intercal Q_j$ are in $\{0, \pm 1\}$, whenever $i \neq j$. 
 \end{proof}

\begin{lemma}\label{lem:matrix-copies}
    For any $p$ and $\varepsilon > 0$, there exists $m_0$ such that for all $m \geq m_0$ we can find $p$ orthogonal $m\times m$ matrices $R_1, \ldots, R_p$ such that  all entries of $R_i^\intercal R_j$ are at most $\varepsilon$ in absolute value for all $i \neq j$.
\end{lemma}
\begin{proof}
    Sample $R_1, \ldots, R_p$ independently from the orthogonal group $O(m)$ with Haar probability measure. Then each column of $R_i$ is a uniform unit vector. By standard concentration of measure on a sphere (see Ball \cite{Bal97} for example), if $u, v$ are uniform unit vectors
    \[
        \PP \Bigl( \ang{u,v} \geq \varepsilon \Bigr) \leq 2 e^{-\varepsilon^2 m/2}.
    \]
    Therefore, by the union bound over all $1\leq s, t \leq m$ and $1 \leq i < j \leq p$,
    \[
        \PP \Bigl( |(R_i^\intercal R_j)_{s,t}| < \varepsilon \text{ for all $1\leq s, t \leq m$ and $1 \leq i < j \leq p$} \Bigr) \geq 1 - 2m^2\binom{p}{2}e^{-\varepsilon m^2/2}.
    \]
    Hence, there exists some $m_0 = m_0(\varepsilon, p)$ such that the probability above is positive whenever $m \geq m_0$. So there exist $R_1, \ldots, R_p$ with desired properties.    
\end{proof}

Let us introduce the templates we will be using in the construction.

\begin{definition}\label{def:uniform-template}
     Let $\alpha, \beta \in (0,1)$. Let $p, s_1, \ldots, s_p$ be positive integers. Write $s = (s_1,\ldots,s_p)$ and define $T_{\alpha, \beta}(p, s)$ by the following block form
         \[
        T_{\alpha, \beta}(p, s) = \begin{pmatrix}
            T_{11} & \cdots & T_{1p} \\
            \vdots & \ddots & \vdots \\
            T_{p1} & \cdots & T_{pp} \\
        \end{pmatrix},
    \]
    where $T_{ij}$ is the $s_i \times s_j$ matrix with all entries equal to $\alpha$ if $i = j$ and $-\beta$ if $i \neq j$. 
\end{definition}

\begin{lemma}\label{lem:psd-template}
    Let $\alpha, \beta \in (0,1)$ and $p$ be positive integer and $s$ be a $p$-tuple of positive integers. If $1 + \alpha/\beta \geq p$, then $T_{\alpha,\beta}(p,s)$ is a positive semidefinite matrix of rank at most $p$. 
    
    Furthermore, if $1 + \alpha/\beta = p$, then the rank of $T_{\alpha,\beta}(p,s)$ is at most $p-1$.
\end{lemma}

\begin{proof}
    First suppose that $s_i = 1$ for all $1 \leq i \leq s$. Then 
    \[
        T_{\alpha, \beta}(p,s) = (\alpha  + \beta)I - \beta J.
    \]
    Since the largest eigenvalue of $J$ is $p$ and $(\alpha + \beta) - \beta p \geq 0$, we see that $T_{\alpha, \beta}(p,s)$ is a positive semidefinite matrix of rank at most $p$. Moreover, if $\alpha + \beta = \beta p$ then $T_{\alpha, \beta}(p,s)$ has rank at most $p-1$. Finally, duplicating rows and columns does not affect the rank or positive semidefiniteness.
\end{proof}

We are now ready to prove \cref{prop:lower}.

\begin{proof}[Proof of \cref{prop:lower} \ref{itm:lower-nint-p}]
Here $\alpha / \beta \notin \ZZ$ and so $\alpha/(p-1) = \beta + \varepsilon$ for some $\varepsilon > 0$. By \cref{lem:matrix-copies} there exists $r_0=r_0(\alpha,\beta)$ such that for each $r \geq r_0$ there exist $r \times r$ orthogonal matrices $R_1, \ldots, R_p$ such that all entries of $R_i^\intercal R_j$ are bounded by $\varepsilon/(1-\alpha)$ in absolute value, whenever $i \neq j$. Let
\[
    Q = (1-\alpha) \begin{pmatrix}
					R_1 & \cdots & R_p
			 \end{pmatrix}^\intercal
                \begin{pmatrix}
					R_1 & \cdots & R_p
			 \end{pmatrix}.
\]
Then, $Q$ is a $pr \times pr$ positive semidefinite matrix of rank at most $r$ with non-diagonal entries bounded in absolute value by $\varepsilon$ and diagonal entries equal to $(1-\alpha)$.

Let $T = T_{\alpha, \alpha/(p-1)}(p,s)$ where $s = (r, \ldots, r) \in \NN^p$ is as in \cref{def:uniform-template}. By \cref{lem:psd-template}, the template $T$ is positive semidefinite of rank at most $p-1$. Let
\[
    M = Q + T.
\]
Then $M$ is a $pr \times pr$ positive semidefinite matrix with rank at most $r + p-1$ such that diagonal entries of $M$ are equal to $1$, and all other entries are either $\alpha$ or at most $\varepsilon - \alpha/(p-1) = -\beta$. Hence by \cref{lem:gram} there exists a spherical $[-1,-\beta] \cup \{\alpha\}$-code with Gram matrix $M$, and so for all $r \geq r_0$,
\[
    N_L(r + p-1) \geq pr.
\]
In particular, there exists $d_0 = d_0(\alpha, \beta)$ such that $N_L(d) \geq p(d-p+1)$ for all $d \geq d_0$.
\end{proof}
\medskip

\begin{proof}[Proof of \cref{prop:lower} \ref{itm:lower-2p}\ref{itm:lower-nint-2p}]
Here $\alpha/\beta \notin \ZZ$ and so $\alpha/(p-1) = \beta + \varepsilon$ for some $\varepsilon > 0$. Also $\lambda \geq 1$. As before, by \cref{lem:matrix-copies} there exists $r_0=r_0(\alpha,\beta)$ such that for each $r \geq r_0$ there exist $r \times r$ orthogonal matrices $R_1, \ldots, R_p$ such that all entries of $R_i^\intercal R_j$ are bounded by $\varepsilon/(1-\alpha)$ in absolute value, whenever $i \neq j$. Let
\[
    Q = (1-\alpha) \begin{pmatrix}
					R_1& -R_1 & \cdots & R_p & -R_p
			      \end{pmatrix}^\intercal
                    \begin{pmatrix}
					R_1 & -R_1 & \cdots & R_p & -R_p
			      \end{pmatrix}.
\]
Then, $Q$ is a $2pr \times 2pr$ positive semidefinite matrix of rank at most $r$ with entries equal to $\pm (1-\alpha)$ or bounded in absolute value by $\varepsilon$.

Suppose $T = T_{\alpha, \alpha/(p-1)}(p,s)$ where $s = (2r, \ldots, 2r) \in \NN^p$. By \cref{lem:psd-template}, $T$ is a positive semidefinite matrix of rank at most $p-1$. Define
\[
    M = Q + T.
\]
Then $M$ is a $2pr \times 2pr$ positive semidefinite matrix with rank at most $r + p-1$ such that diagonal entries of $M$ are equal to $1$, and all other entries are either $\alpha$ or $\alpha - (1-\alpha) \leq -\beta$ (this is equivalent to $\lambda \geq 1$) or at most $\varepsilon - \alpha/(p-1) = -\beta$. Hence by \cref{lem:gram} there exists a spherical $[-1,-\beta] \cup \{\alpha\}$-code with Gram matrix $M$, so for all $r \geq r_0$,
\[
    N_L(r + p-1) \geq 2pr.
\]
In particular, there exists $d_0 = d_0(\alpha, \beta)$ such that $N_L(d) \geq 2p(d-p+1)$ for all $d \geq d_0$.
\end{proof}

\begin{proof}[Proof of \cref{prop:lower} \ref{itm:lower-2p}\ref{itm:lower-int-2p}]
Suppose $\alpha/\beta \in \ZZ$ with $\alpha/\beta \geq 2$ and $\lambda \in \ZZ$. In particular, this means $\beta = \alpha/(p-1)$ with $p \geq 3$. By \cref{prop:matrix-copies} there exist an integer $m$ and $m\times m$ matrices $Q_1, \ldots, Q_p$ such that $Q_i^\intercal Q_i = \lambda I$ and $Q_i^\intercal Q_j$ has entries in $\{0, \pm 1\}$ whenever $i \neq j$. For any positive integer $r$, define an $rm \times rm$ matrix $R^{(r)}_i$ by 
\[
    R_i^{(r)} = \lambda^{-1/2}\begin{pmatrix}
                Q_i & & \\
                & \ddots & \\
                & & Q_i \\
                \end{pmatrix}.
\]
There are $r$ copies of $Q_i$ on the diagonal. Let
\[
    Q = (1-\alpha) \begin{pmatrix}
					R_1^{(r)} & -R_1^{(r)} & \cdots & R_p^{(r)} & -R_p^{(r)}
			      \end{pmatrix}^\intercal
                    \begin{pmatrix}
					R_1^{(r)} & -R_1^{(r)} & \cdots & R_p^{(r)} & -R_p^{(r)}
			      \end{pmatrix}.
\]
Then, $Q$ is  a $2prm \times 2prm$ positive semidefinite matrix of rank at most $rm$ with entries equal to $\pm (1-\alpha)$ or an element in $\{0, \pm(1-\alpha)\lambda^{-1}\} = \{0, \pm(\alpha + \beta)\}$.

Let $T = T_{\alpha, \beta}(p,s)$ where $s = (2mr, \ldots, 2mr) \in \NN^p$. By \cref{lem:psd-template} the template $T$ is positive semidefinite of rank at most $p-1$. Let
\[
    M = Q + T.
\]

Then $M$ is a $2prm \times 2prm$ positive semidefinite matrix with rank at most $rm + p-1$ such that diagonal entries of $M$ are equal to $1$, and all other entries are either $\alpha$ or $\alpha - (1-\alpha) \leq -\beta$ (this is equivalent to $\lambda \geq 1)$ or $\{0, \pm(\alpha + \beta)\}-\beta \subseteq [-1,-\beta]\cup \{\alpha\}$. Hence by \cref{lem:gram} there exists a spherical $[-1,-\beta] \cup \{\alpha\}$-code with Gram matrix $M$. So 
\[
    N_L(rm + p-1) \geq 2prm.
\]
In particular, if $d - p + 1$ is divisible by $m$ then $N_L(d) \geq 2p(d-p+1)$.
\end{proof}
\medskip

\begin{proof}[Proof of \cref{prop:lower} \ref{itm:lower-2p}\ref{itm:lower-4}]
Here $\alpha = \beta$ and $\lambda^2 \in \ZZ$. By \cref{prop:matrix} there exists an integer $m$ and a $m \times m$ orthogonal matrix $R$ with entries in $\{0, \pm \lambda^{-1}\}$. For any positive integer $r$, define $rm \times rm$ matrix $R^{(r)}$ by 
\[
    R^{(r)} =  \begin{pmatrix}
                R & & \\
                & \ddots & \\
                & & R \\
                \end{pmatrix}.
\]
There are $r$ copies of $R$ on the diagonal. Let
\[
			 Q = (1-\alpha)
			 \begin{pmatrix}
					 I & - I & R^{(r)} & -R^{(r)}
			 \end{pmatrix}^\intercal
			 \begin{pmatrix}
				     I & - I & R^{(r)} & -R^{(r)}
			 \end{pmatrix}
\]
Then $Q$ is a $4mr \times 4mr$ positive semidefinite matrix of rank at most $rm$ with entries equal to $\pm(1-\alpha)$ or an element in $\{\pm(1-\alpha)\lambda^{-1}\} = \{0, \pm2 \alpha\}$.

Suppose $T = T_{\alpha, \alpha}(2,s)$ where $s = (2mr, 2mr)$. By \cref{lem:psd-template} the template $T$ is positive semidefinite of rank $1$. Define
\[
    M = Q + T.
\]
Then $M$ is a $4rm \times 4rm$ positive semidefinite matrix with rank at most $rm + 1$ such that diagonal entries of $M$ are equal to $1$, and all other entries are either $\alpha$ or $\alpha - (1-\alpha) \leq -\beta$ (this is equivalent to $\lambda \geq 1$) or $\{0,\pm2\alpha\}-\alpha \subseteq [-1,-\alpha]\cup \{\alpha\}$. Hence by \cref{lem:gram} there exists a spherical $[-1,-\alpha] \cup \{\alpha\}$-code with Gram matrix $M$. So 
\[
    N_L(rm +1) \geq 4rm.
\]
In particular, if $d - 1$ is divisible by $m$ then $N_L(d) \geq 4(d-1)$.
\end{proof}

Next, we show that when $\alpha/\beta$ is near the upper end of its allowed interval $[p-1,p)$, one can add additional vectors to the $L$-code.

\begin{proposition}\label{prop:lower-extra}
     Let $L = [-1,-\beta]\cup \{\alpha\}$ with $\alpha, \beta \in (0,1)$ and $\lambda = (1-\alpha)/(\alpha + \beta)$. Assume $\lambda \ge 1$.
     Then, for any positive integers $p$ and $M$ we can find $\varepsilon > 0$ and $d_0$ such that if $p-\varepsilon < \alpha/\beta < p$, then for all $d \geq d_0$,
    \[
         N_L(d) \geq 2pd + M.
    \]
\end{proposition}

\begin{proof}
    Take $K = M + 2p(p+2)$ and $\varepsilon = 1/(pK + K)$. Suppose $p-\varepsilon < \alpha/\beta < p$. Take $\gamma = \alpha/(p-\varepsilon)$ (this will be a proxy for $\beta$) and note that the choice of $\varepsilon$ guarantees 
    \[
      1 - \frac{\varepsilon \gamma (p+1) K}{1-\alpha} = 1 - \frac{\gamma}{1-\alpha} > 0
    \]
    The last inequality follows from $\lambda \geq 1$. Hence the matrix given by the difference 
    \[
        I_{K}- \frac{\varepsilon (p+1) \gamma}{1-\alpha} J_{K}.
    \]
    is positive semidefinite. In particular, we can find a $K \times K$ matrix $B$ such that
    \[
        B^\intercal B = I_{K}- \frac{\varepsilon (p+1) \gamma}{1-\alpha} J_{K}.
    \]
    For $r \geq K$, extend $B$ with zeroes to form a $r \times K$ matrix $B_r$ such that $B_r^\intercal B_r = B^\intercal B$.
    
    By \cref{lem:matrix-copies}, there exists $r_0 \geq K$ such that for all $r \geq r_0$ we can find $r \times r$ orthogonal matrices $R_0, \ldots, R_{p}$ such that for all $i \neq j$, all entries of
    \[
        (1-\alpha)R_i^\intercal R_j \quad \text{ and } \quad (1-\alpha)R_i^\intercal R_jB_r
    \]
    are at most $(\gamma - \beta)/(1-\alpha)$ in absolute value. For $r \geq r_0$ define
    \[
    Q = (1-\alpha) \begin{pmatrix}
					 R_{0}B_r & R_1 & -R_1 & \cdots & R_p & -R_p
			      \end{pmatrix}^\intercal
                    \begin{pmatrix}
					R_{0}B_r & R_1 & -R_1 & \cdots & R_p & -R_p 
			      \end{pmatrix}.
    \]
This defines a $(K + 2pr) \times (K + 2pr)$ positive semidefinite matrix of rank at most $r$. Note that the top-left  $K \times K$ block of $Q$ is given by $ (1-\alpha)I_{K} - \varepsilon (p+1) J_{K}$. Remaining entries are either equal to $\pm(1-\alpha)$ or bounded in absolute value by $\gamma - \beta$.
    
Define $T = T_{\alpha, \gamma}(p+1, s)$ where $s = (K, r, \ldots, r) \in \NN^{2p+1}$. This matrix itself is not positive semidefinite. To fix this, we define the $(K + 2pr) \times (K + 2pr)$ matrix $E$ by extending $\varepsilon \gamma(p+1) J_{K}$ with zeroes.

As in the proof of \cref{lem:psd-template}, to show that $T+E$ is positive semidefinite, it suffices to show the positive semidefiniteness of
\[
    (\alpha + \gamma)I_{p+1} - \gamma J_{p+1} + \varepsilon(p+1)\gamma e_1 e_1^\intercal. 
\]
Since $(\alpha + \gamma)I_{p+1} - \gamma J_{p+1}$ has a single negative eigenvalue with eigenvector $\pmb{1}$, it suffices to test the positive semidefiniteness against the vector $\pmb{1}$. This yields
\[
  \pmb{1}^\intercal \Bigl[(\alpha + \gamma)I_{p+1} - \gamma J_{p+1} + \varepsilon(p+1)\gamma e_1 e_1^\intercal\Bigr] \pmb{1}=  (\alpha + \gamma)(p+1) - \gamma(p+1)^2 + \varepsilon (p+1) \gamma = 0.
\]
Now take $M = Q + T + E$. Clearly $M$ is a $(K + 2pr) \times (K + 2pr)$ positive semidefinite matrix of rank at most $(r + p + 2)$. Moreover, observe that the top-left $K \times K$ block of $M$ is given by
\[
    \Bigl[(1-\alpha)I_{K} - \varepsilon (p+1) J_{K}\Bigr] + \alpha J_{K} + \varepsilon (p+1) \gamma J_{K} =  (1-\alpha)I_{K} + \alpha J_{K}.
\]
Hence, as earlier, the diagonal entries of $M$ are equal to $1$ and all remaining entries are either $\alpha$ or $\alpha - (1-\alpha) \leq -\beta$ or at most $-\gamma + (\gamma - \beta) \leq -\beta$.

Therefore, by \cref{lem:gram} we see that $M$ is a Gram matrix of an $L$-code of size $2pr + K$ in $\RR^{r+p+2}$ whenever $r \geq r_0$. Taking $r = d-p-2$ gives $N_L(d) \geq 2pd + M$.
\end{proof}

\section{Global Structure Theorem}\label{sec:global}

In this section, we will prove \cref{thm:structure}. 

Recall that given a spherical code $C$, its \emph{Gram graph} $G_C$ is the edge-weighted complete graph on vertex set $C$ such that edge $uv$ has weight $\ang{u,v}$. The \emph{negative graph} $G_C^-$ is the simple graph consisting of negative edges of $G_C$ (and then forgetting the edge weights). Likewise the \emph{positive graph} $G_C^+$ consisting of positive edges of $G_C$ (and then forgetting the edge weights).
We shall talk about the negative and positive degrees to refer to degrees in $G_C^-$ and $G_C^+$. For example, given $A \subseteq C$, the \emph{maximum positive degree in $A$} is the maximum degree of $G_A^+$. Given $A,B \subseteq C$, the \emph{maximum negative degree from $B$ to $A$} is the maximum number of neighbors that $b \in B$ can have in $A$ in the negative graph $G_C^-$.

\subsection{Forbidden local structures}

The next three lemmas use the positive semidefiniteness of the Gram matrix to deduce forbidden induced subgraphs in the Gram graph.

\begin{lemma}\label{lem:negative-clique}
The largest negative clique in any spherical $[-1,-\beta]\cup\{\alpha\}$ has at most $\beta^{-1} + 1$ vertices.
\end{lemma}

\begin{proof}
Let $V$ be the set of vectors corresponding to the largest negative clique and $M_V$ be the Gram matrix of $V$. Then,
\[
	0 \leq \mathbf{1}_V^\intercal M_V\mathbf{1}_V \leq |V| - \beta |V| (1+|V|).
\]
Rearranging the terms yields $|V| \leq \beta^{-1} + 1$.
\end{proof}

\begin{lemma}\label{lem:very-strong-positive-clique}
Let $ \alpha, \beta \in(0,1)$ and $L \subseteq [-1,-\beta]\cup\{\alpha\}$. Then for all $\Delta_0$, there exists  $K=K(\alpha,\beta,\Delta_0)$ such that the following holds for all $\Delta \geq 1$: 

Let $C = A \sqcup B \sqcup \{v\}$ be a spherical $L$-code such that for some $\gamma \in [-1,1]$,
\begin{itemize}
    \item[(a)] maximum negative degree from $B$ to $A$ is at most $\Delta_0$, and
    \item[(b)] all edges from $v$ to $A$ have weight at least $\gamma + K\Delta^{-1/2}$, and
    \item[(c)] all edges from $v$ to $B$ have weight at most $\gamma - K\Delta^{-1/2}$.
\end{itemize}
Then either $|A| < \Delta$ or $|B| < \Delta$.
\end{lemma}

\begin{figure}[h]
\definecolor{ffzzcc}{rgb}{1,0,1}
\definecolor{qqzzff}{rgb}{0,0.6,1}
\definecolor{wwccff}{rgb}{0.4,0.8,1}
\begin{tikzpicture}

\fill[line width=0pt,color=wwccff,fill=wwccff!30] 
(-1.5, 0.5) rectangle (-0.5, -0.5);
\node at (-1, 0) {\tiny $\alpha$};

\fill[line width=0pt,color=wwccff,fill=wwccff] 
(-1.5, 1.5) rectangle (-0.5, 0.5);
\node at (-1, 1) {\small $A$};

\fill[line width=0pt,color=wwccff,fill=wwccff] 
(-1.5, -1.5) rectangle (-0.5, -0.5);
\node at (-1, -1) {\small $B$};

\draw [fill=black] (1,0) circle (2pt)  node[below right] {\large$v$};

\fill[line width=0pt,color=wwccff,fill=wwccff!30]  (1, 0) -- (-0.5, 1.5) -- (-0.5, 0.5) -- cycle;
\node[rotate=-30] at (0,.6) {\tiny $> \gamma + \varepsilon$};

\fill[line width=0pt,color=wwccff,fill=wwccff!30]  (1, 0) -- (-0.5, -0.5) -- (-0.5, -1.5) -- cycle;
\node[rotate=30] at (0,-.6) {\tiny $< \gamma - \varepsilon$};

\end{tikzpicture}
\caption{Forbidden configuration from \cref{lem:very-strong-positive-clique}.}
\end{figure}
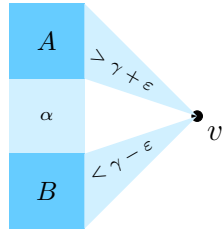

\begin{proof}[Proof of \cref{lem:very-strong-positive-clique}]
Let $K = \sqrt{1 + \Delta_0}$. Suppose the contrary that $C = A \cup B \cup \{v\}$ is a spherical $L$-code satisfying the conditions $(a), (b), (c)$ for some $\Delta \geq 1$ and $\gamma \in [-1,1]$, such that $|A| \geq \Delta $ and $|B| \geq \Delta$. Note that
\begin{align*}
	\left(\frac{\mathbf{1}_A}{|A|} - \frac{\mathbf{1}_B}{|B|}\right)^\intercal M_C \left(\frac{\mathbf{1}_A}{|A|} - \frac{\mathbf{1}_B}{|B|}\right) &= \frac{1}{|A|^2}\mathbf{1}_A^\intercal M_C \mathbf{1}_A + \frac{1}{|B|^2}\mathbf{1}_B^\intercal M_C \mathbf{1}_B - \frac{2}{|A||B|}\mathbf{1}_A^\intercal M_C \mathbf{1}_B  \\
	&\leq \left(\alpha + \frac{1}{|A|}\right) + \left(\alpha + \frac{1}{|B|}\right) - 2\left(\alpha - \frac{2\Delta_0}{|A|}\right) \\
        &\leq \left(\frac{2 + 4\Delta_0 }{\Delta}\right).
\end{align*}
Let $w = t\mathbf{1}_v - \mathbf{1}_A/|A| + \mathbf{1}_B/|B|$ for some $t>0$. Since the Gram matrix $M_C$ is positive semidefinite,
\begin{align*}
	0 \leq w^\intercal M_Cw &\leq t^2  - 2t\left(\gamma + K\Delta^{-1/2}\right) + 2t\left(\gamma -K \Delta^{-1/2}\right) + \frac{2 + 4\Delta_0 }{\Delta} \\
 &= \left(t - \frac{2K}{\sqrt{\Delta}}\right)^2 + \frac{2 + 4\Delta_0 - 4K^2}{\Delta}.
\end{align*}
Since $K = \sqrt{1 + \Delta_0}$, the inequality fails for $t = 2K\Delta^{-1/2}$. This is a contradiction.
\end{proof}

The next lemma roughly states that if there are $q$ not-too-small positive cliques, with mostly negative edges between each pair of cliques, then $q \le \floor{\alpha/\beta} + 1$.

\begin{lemma}\label{lem:many-cliques}
    Let $\alpha, \beta \in (0,1)$ and $L \subseteq [-1,-\beta]\cup\{\alpha\}$. For all $\Delta_0$ there exists $P = P(\alpha, \beta, \Delta_0)$ such that if $A_1 \sqcup \ldots \sqcup A_q$ is a spherical $L$-code satisfying 

    \begin{itemize}
        \item[(a)] for all $1 \leq i \leq q$, the $L$-code $A_i$ is a positive clique of size $P$;
    
        \item[(b)] for all $1 \leq i < j \leq q$, the maximum positive degree from $A_i$ to $A_j$ is at most $\Delta_0$,
    \end{itemize}
    then $q \leq \lfloor \alpha/\beta\rfloor + 1$.
\end{lemma}

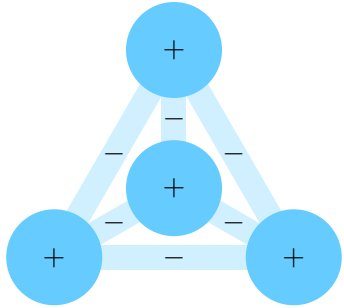
\begin{figure}[h]
\definecolor{wwccff}{rgb}{0.4,0.8,1}
\begin{tikzpicture}[scale = 0.6]
\fill[color = wwccff!30] (-0.7524102376159736,2.3116585248889) -- (-2.3781601262366365,-0.5042228826015307) -- (-1.901151149655595,-0.7796241436364628) -- (-0.2754012610349321,2.036257263853968) -- cycle;
\node at (-1.3267806936357842,0.7660171906262186) {$-$};

\fill[color = wwccff!30] (0.2754012610349319,2.0362572638539675) -- (0.7524102376159734,2.3116585248888994) -- (2.378160126236635,-0.5042228826015329) -- (1.9011511496555942,-0.7796241436364647) -- cycle;
\node at (1.3267806936357835,0.7660171906262179) {$-$};

\fill[color = wwccff!30] (-1.6257498886206623,-1.2566331202175038) -- (-1.6257498886206623,-1.807435642287369) -- (1.6257498886206634,-1.8074356422873683) -- (1.6257498886206634,-1.2566331202175054) -- cycle;
\node at (0,-1.5320343812524366) {$-$};

\fill[color = wwccff!30] (-1.6257498886206623,-1.2566331202175038) -- (-0.6948350198094596,-0.7191691701167311) -- (-0.9702362808443916,-0.24216019353568968) -- (-1.901151149655595,-0.7796241436364628) -- cycle;
\node at (-1.3267806936357842,-0.766017190626218) {$-$};

\fill[color = wwccff!30] (-0.2754012610349321,0.9613293636524214) -- (0.2754012610349319,0.9613293636524213) -- (0.2754012610349319,2.0362572638539675) -- (-0.2754012610349321,2.036257263853968) -- cycle;
\node at (0,1.5320343812524366) {$-$};

\fill[color = wwccff!30] (1.9011511496555942,-0.7796241436364647) -- (1.6257498886206634,-1.2566331202175054) -- (0.6948350198094596,-0.7191691701167312) -- (0.9702362808443913,-0.24216019353569118) -- cycle;
\node at (1.3267806936357835,-0.7660171906262188) {$-$};

\fill [color = wwccff] (0,0) circle (1.064068762504873cm);
\node at (0,0) {$+$};

\fill [color = wwccff]  (0,3.0640687625048733) circle (1.0640687625048733cm);
\node at (0,3.0640687625048733) {$+$};

\fill [color = wwccff]  (-2.6535613872715684,-1.532034381252436) circle (1.0640687625048733cm);
\node at (-2.6535613872715684,-1.532034381252436) {$+$};

\fill [color = wwccff]  (2.653561387271567,-1.5320343812524375) circle (1.0640687625048733cm);
\node at (2.653561387271567,-1.5320343812524375) {$+$};
\end{tikzpicture}

\caption{Forbidden structure from \cref{lem:many-cliques}.}
\end{figure}

\begin{proof}
	We have
	\begin{align*}
		0 \leq \mathbf{1}_C^\intercal M_C \mathbf{1}_C &= \sum_{i = 1}^q\mathbf{1}_{A_i}^\intercal M_C \mathbf{1}_{A_i} + 2\sum_{1\leq i < j \leq q}\mathbf{1}_{A_i}^\intercal M_C \mathbf{1}_{A_j} \\
        &\leq  q\left(\alpha P^2 + P\right) + q(q-1)\left(-\beta P^2 + 2 \Delta_0 P\right).
    \end{align*}
Rearranging, we see that
\[
	q \leq 1 + \frac{\alpha + 1/P}{\beta - 2\Delta_0/P}.
	\]
Choose $P$ such that the right hand side is strictly less than $2 + \floor{\alpha/\beta}$. Then $q \leq 1 + \floor{\alpha/\beta}$.
\end{proof}

\subsection{Deducing global structure}
In the remainder of this section, we deduce the global structure theorem, \cref{thm:structure}.

\begin{lemma}\label{lem:strong-positive-clique}
Let $\alpha, \beta \in (0,1)$ and $L \subseteq [-1,-\beta]\cup\{\alpha\}$.
For all $\Delta_0$ there exists $\Delta$ such that if $C = A \sqcup B \sqcup \{v\}$ is an $L$-code satisfying 
\begin{itemize}
    \item[(a)] the maximum negative degree from $B$ to $A$ is at most $\Delta_0$, and
    \item[(b)] all edges from $v$ to $A$ have positive weight, and
    \item[(c)] all edges from $v$ to $B$ have negative weight.
\end{itemize}
Then either $|A| \leq \Delta$ or $|B| \leq \Delta$.
\end{lemma}

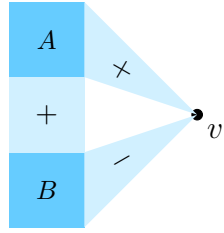
\begin{figure}[h]
\definecolor{ffzzcc}{rgb}{1,0,1}
\definecolor{qqzzff}{rgb}{0,0.6,1}
\definecolor{wwccff}{rgb}{0.4,0.8,1}
\begin{tikzpicture}
\fill[line width=0pt,color=wwccff,fill=wwccff] 
(-1.5, 1.5) rectangle (-0.5, 0.5);
\node at (-1, 1) {\small $A$};

\fill[line width=0pt,color=wwccff,fill=wwccff] 
(-1.5, -1.5) rectangle (-0.5, -0.5);
\node at (-1, -1) {\small $B$};

\fill[line width=0pt,color=wwccff,fill=wwccff!30] 
(-1.5, 0.5) rectangle (-0.5, -0.5);
\node at (-1, 0) {$+$};

\draw [fill=black] (1,0) circle (2pt)  node[below right] {\large$v$};

\fill[line width=0pt,color=wwccff,fill=wwccff!30]  (1, 0) -- (-0.5, 1.5) -- (-0.5, 0.5) -- cycle;
\node[rotate=-30] at (0,.6) { $+ $};

\fill[line width=0pt,color=wwccff,fill=wwccff!30]  (1, 0) -- (-0.5, -0.5) -- (-0.5, -1.5) -- cycle;
\node[rotate=30] at (0,-.6) {$-$};

\end{tikzpicture}
\caption{Forbidden configuration from \cref{lem:strong-positive-clique}.}
\end{figure}

\begin{proof}
	Take $\gamma=0$ and $\Delta = \lceil K^2/\min(\alpha,\beta)^2 \rceil$ in \cref{lem:very-strong-positive-clique}.
\end{proof}

The next lemma states roughly that when given a large positive clique, any other vertex is connected to the clique either mostly via positive edges or mostly via negative edges.

\begin{lemma}\label{lem:positive-clique}
Let $\alpha, \beta \in (0,1)$ and $L \subseteq [-1,-\beta]\cup\{\alpha\}$. For all $\Delta_0$ there exists $\Delta$ such that if $S\sqcup\{v\}$ is an $L$-code where maximum negative degree in $S$ is at most $\Delta_0$, then either
\begin{enumerate}
    \item[(a)] the negative degree from $v$ to $S$ is at  most $\Delta$, or
    \item[(b)]  the positive degree from $v$ to $S$ is at  most $\Delta$.
\end{enumerate}
\end{lemma}

\begin{proof}
	Let $A$ and $B$ to be the set of positive and negative neighbors of $v$ in $S$, respectively. Take $\gamma=0$ and $\Delta = \lceil K^2/\min(\alpha,\beta)^2 \rceil$ in \cref{lem:very-strong-positive-clique}.
\end{proof}

The following lemma is a more general version of \cref{lem:positive-clique}. It states that most edges from a vertex to a large positive clique have weights concentrated around a single value.

\begin{lemma}\label{lem:edges-to-positive-clique}
Let $ \alpha, \beta \in (0,1)$ and $L \subseteq [-1,-\beta]\cup\{\alpha\}$. For all $\Delta_0$ there exists $K = K(L, \Delta_0)$ such that the following holds for all $\Delta \geq 1$:

If $S\sqcup\{v\}$ is an $L$-code where the maximum negative degree in $S$ is at most $\Delta_0$, then there exists $\gamma \in \RR$ such that for all but at most $2\Delta$ many $s \in S$
\[
	|\gamma - \langle s, v \rangle| \leq K \Delta^{-1/2}.
\]
\end{lemma}

\begin{proof}
	Let $K$ be the constant from \cref{lem:very-strong-positive-clique}. For $\theta \in \RR$ define 	
\[
		S_+(\theta) \coloneqq \left\{s \in S : \langle s, v\rangle \geq \theta + K\Delta^{-1/2}\right\}
  \]
  and
  \[
		S_-(\theta) \coloneqq \left\{s \in S : \langle s, v\rangle \leq \theta - K\Delta^{-1/2}\right\}.
  \]
	By \cref{lem:very-strong-positive-clique} applied to $S_-(\theta) \sqcup S_+(\theta) \sqcup \{v\}$ we see that $\min\set{|S_-(\theta)|, |S_+(\theta)|} \leq \Delta$. Define
 \[
     \theta^* = \max \set{\theta \in \RR : |S_+(\theta)| > \Delta} \text{ and } \theta_* = \min \set{\theta \in \RR : |S_-(\theta)| > \Delta}.
 \]
 Then, $\theta_* > \theta^*$ or else if $\theta \in [\theta_*, \theta^*]$, then $\min\set{|S_-(\theta)|, |S_+(\theta)|} > \Delta$, which is a contradiction. Hence, if we take $\gamma \in (\theta^*, \theta_*)$ we have $ |S_-(\gamma)| \leq \Delta$ and $|S_+(\gamma)| \leq \Delta$. In other words, for all but at most $2\Delta$ many $s \in S$ we have $|\gamma - \ang{s,v}| \leq K \Delta^{-1/2}$. 
\end{proof}

The next lemma says that, given two mostly positive cliques, most edges between them have roughly equal weights.

\begin{lemma}\label{lem:two-positive-cliques}
Let $ \alpha, \beta \in (0,1)$ and $L \subseteq [-1,-\beta]\cup\{\alpha\}$. Then, for every $\Delta_0$ there exists a constant $K = K(L, \Delta_0) $ such that the following holds for all $\Delta \geq K$: 

Suppose $A \sqcup B$ is an $L$-code with $\min\set{|A|,|B|} \geq \Delta^2$ such that the maximum negative degrees in $A$ and in $B$ are at most $\Delta_0$. Then, there exist some $\gamma \in \RR$ and $R \subseteq A \cup B$  of size at most $\Delta$ such that
\begin{enumerate}
	\item for each $a \in A \setminus R$ and all but at most $\Delta$ many $b \in B$,
	\[
		|\gamma - \langle a,b\rangle| \leq K\Delta^{-1/2};
	\]
	\item for each $b \in B \setminus R$ and all but at most $\Delta$ many $a \in B$,
	\[
		|\gamma - \langle a,b\rangle| \leq K\Delta^{-1/2}.
	\]
\end{enumerate}
\end{lemma}	

\begin{proof}
Let $K_1$ be a constant greater than that of \cref{lem:edges-to-positive-clique} and \cref{lem:very-strong-positive-clique}. To guarantee that all estimates below are trivially satisfied, take $K = 2^{200}K_1$. Furthermore, for any $\Delta \geq K$ take $\Delta_1 = 2^{-100}\Delta$. 

According to \cref{lem:edges-to-positive-clique}, for every $a \in A$, there exists a corresponding $\gamma_a \in L$, such that all but at most $\Delta_1$ elements $b \in B$ satisfy $|\langle a, b \rangle - \gamma_a | \leq K_1\Delta_1^{-1/2}$. Define  $\gamma_b$ similarly for each $b \in B$.

We will first now demonstrate the existence of $\gamma_A \in \RR$ such that for all but at most $2\Delta_1$ elements $a \in A$ we have
$|\gamma_a - \gamma_A| \leq 4K_1\Delta_1^{-1/2}$. For each $\theta \in \RR$ set 	
\[
		S_+(\theta) \coloneqq \left\{a \in A : \gamma_a \geq \theta+ 4K_1\Delta_1^{-1/2}\right\}
  \]
  and
  \[
		S_-(\theta) \coloneqq\left\{a \in A : \gamma_a \leq \theta-  4K_1\Delta_1^{-1/2}\right\}.
\]
Let us first show $\min\set{|S_+(\theta)|, |S_-(\theta)|} \leq  \Delta_1$. Suppose the contrary that $|S_+(\theta)| \geq (\Delta_1 + 1)$ and $|S_-(\theta)| \geq  (\Delta_1 +1)$. Let $S_+'(\theta) \subseteq S_+(\theta)$ be arbitrary subset with $(\Delta_1 + 1)$ elements. Likewise, let $S_-'(\theta) \subseteq S_-(\theta)$ be arbitrary subset with $(\Delta_1 + 1)$ elements. Since $|B| > 2\Delta_1(\Delta_1+1)$, we can find an $v \in B$ such that for all $a \in S_+'(\theta) \cup S_-'(\theta)$ 
\[
		|\langle a,v \rangle -\gamma_{a}| \leq K_1\Delta_1^{-1/2}.
\]
By \cref{lem:very-strong-positive-clique} applied to $S_-(\theta)' \cup S_+'(\theta) \cup \{v\}$, we get a contradiction.

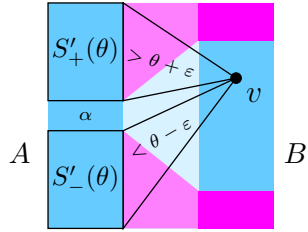
\begin{figure}[h]\label{fig:two-positive-cliques-1}
\definecolor{ffzzcc}{rgb}{1,0,1}
\definecolor{qqzzff}{rgb}{0,0.6,1}
\definecolor{wwccff}{rgb}{0.4,0.8,1}
\begin{tikzpicture}
\fill[line width=0pt,color=wwccff,fill=wwccff] (1.5,1.5) -- (0.5,1.5) -- (0.5,-1.5) -- (1.5,-1.5) -- cycle;
\fill[line width=0pt,color=wwccff,fill=wwccff] (-0.5,1.5) -- (-1.5,1.5) -- (-1.5,-1.5) -- (-0.5,-1.5) -- cycle;

\fill[line width=0pt,color=wwccff,fill=wwccff!25] (0.5,1.5) -- (-0.5,1.5) -- (-0.5,-1.5) -- (0.5,-1.5) -- cycle;


\draw[line width=0.5pt] (-0.5,0.2) -- (-1.5,0.2) -- (-1.5,1.5) -- (-0.5,1.5) -- cycle;
\draw[line width=0.5pt] (-0.5,-1.5) -- (-0.5,-0.2) -- (-1.5,-0.2) -- (-1.5,-1.5) -- cycle;


\fill[line width=0pt,color=ffzzcc,fill=ffzzcc!50] (-0.5,-0.2) -- (0.5,-1) -- (1.5,-1) -- (1.5,-1.5) -- (0.5,-1.5) -- (-0.5,-1.5) -- cycle;
\fill[line width=0pt,color=ffzzcc,fill=ffzzcc!50] (-0.5,0.2) -- (-0.5,1.5) -- (0.5,1.5) -- (1.5,1.5) -- (1.5,1) -- (0.5,1) -- cycle;
\fill[line width=0pt,color=ffzzcc,fill=ffzzcc] (0.5,1.5) -- (1.5,1.5) -- (1.5,1) -- (0.5,1) -- cycle;
\fill[line width=0pt,color=ffzzcc,fill=ffzzcc]  (0.5,-1.5) -- (1.5,-1.5) -- (1.5,-1) -- (0.5,-1) -- cycle;

\draw [fill=black] (1,0.5) circle (2pt)  node[below right] {\large$v$};
\draw [line width = 0.5pt] (1, 0.5) -- (-0.5, 1.5) -- (-0.5, 0.2) -- cycle;
\draw [line width = 0.5pt] (1, 0.5) -- (-0.5, -0.2) -- (-0.5, -1.5) -- cycle;

\node[rotate=-15] at (0, 0.7) {\tiny $> \theta + \varepsilon$};
\node[rotate=30] at (0, -0.3) {\tiny $< \theta - \varepsilon$};
\node at (-1, 0) {\tiny $\alpha$};
\node at (-1, 0.85) {\small $S_+'(\theta)$};
\node at (-1, -0.85) {\small $S_-'(\theta)$};
\node[left] at (-1.6, -0.5) {$A$};
\node[right] at (1.5, -0.5) {$B$};
\end{tikzpicture}
\caption{Illustration of the argument that nearly all $\gamma_a \approx \gamma_A$ in \cref{lem:two-positive-cliques}. We find $v \in B$ whose edges to $S_+'(\theta)$ have large weights and to $S_-'(\theta)$ have small weights. By \cref{lem:very-strong-positive-clique} at least one of $S_+'(\theta)$ and $S_-'(\theta)$ has small size. Thus, we can find $\theta = \gamma_A$ so that both sets have small size.}
\end{figure}
Next, we proceed as in the proof of \cref{lem:edges-to-positive-clique}. Define
\[
    \theta^* = \max\set{\theta  \in \RR: |S_+(\theta)| > \Delta_1} \text{ and } \theta_* = \max\set{\theta  \in \RR: |S_-(\theta)| > \Delta_1}.
\]
Then, $\theta_* > \theta^*$ or else if $\theta \in [\theta_*, \theta^*]$, then $\min\set{|S_-(\theta)|, |S_+(\theta)|} > \Delta_1$, which is a contradiction. Hence, if we select $\gamma_A \in (\theta^*, \theta_*)$ we get $ |S_-(\gamma)| \leq \Delta_1$ and $|S_+(\gamma)| \leq \Delta_1$. In other words, for all but at most $2\Delta_1$ elements $a \in A$, we have $|\gamma_a - \gamma_A| \leq 4K_1\Delta_1^{-1/2}$. Similarly, we can select $\gamma_B \in \RR$ such that for all but $2\Delta_1$ elements $b \in B$, we have $|\gamma_b - \gamma_B| \leq 4K_1\Delta_1^{-1/2}$. 
	
We now want to show that $\gamma_A$ and $\gamma_B$ are close. Since $|A| > (\Delta_1+1) + (2\Delta_1)$ and $|B| > (\Delta_1 + 1)\Delta_1 + (2\Delta_1)$  we can find a subset $A' \subseteq A$ of size $(\Delta_1 + 1)$ and a subset $B'$ of size larger than $\Delta_1(\Delta_1 + 1)+1$ such that for all $a \in A'$ and $b \in B'$
\[
	|\gamma_a - \gamma_A| \leq 4K_1\Delta_1^{-1/2} \quad \text{ and } \quad |\gamma_b - \gamma_B| \leq 4K_1\Delta_1^{-1/2}.
\]
Hence, we can find $y \in B'$ with $|\gamma_y - \ang{a,y}| \leq K_1\Delta_1^{-1/2}$ for all $a \in A'$. Then, we can find $x \in A'$ such that $|\gamma_x - \ang{x, b}| \leq K_1 \Delta_1^{-1/2}$ for all $b \in B'$. Thus,
	\begin{align*}
		|\gamma_A - \gamma_B| &\leq |\gamma_A-\gamma_x|+|\gamma_x-\langle x, y\rangle
		| + |\langle x, y\rangle - \gamma_y| + |\gamma_y - \gamma_B| \\
	          &\leq 4K_1\Delta_1^{-1/2} +  K_1\Delta_1^{-1/2} +  K_1\Delta_1^{-1/2}+4K_1\Delta_1^{-1/2} \leq 10K_1\Delta_1^{-1/2}.
	\end{align*}
Choosing $\gamma = (\gamma_A + \gamma_B)/2$, we see that for all but at most $2\Delta_1$ elements $a \in A$, we have
\[
    |\gamma_a  -  \gamma| \leq |\gamma_a-\gamma_A| + |\gamma_A - \gamma| \leq 4K_1\Delta_1^{-1/2} + 5K_1\Delta_1^{-1/2} = 9 K_1\Delta_1^{-1/2}.
\]
Likewise, for all but at most $2\Delta_1$ elements $b \in B$, we have $|\gamma_b - \gamma| \leq 9K_1\Delta_1^{-1/2}$. 
\end{proof}

\begin{figure}\label{fig:two-positive-cliques-2}
   \definecolor{ffzzcc}{rgb}{1,0,1}
   \definecolor{wwccff}{rgb}{0.4,0.8,1}
\begin{tikzpicture}
\fill[line width=0pt,color=black,fill=wwccff] (1.5,1.5) -- (0.5,1.5) -- (0.5,-1.5) -- (1.5,-1.5) -- cycle;
\fill[line width=0pt,color=wwccff,fill=wwccff] (-0.5,1.5) -- (-1.5,1.5) -- (-1.5,-1.5) -- (-0.5,-1.5) -- cycle;
\fill[line width=0pt,color=wwccff,fill=wwccff!40] (0.5,1.5) -- (-0.5,1.5) -- (-0.5,-1.5) -- (0.5,-1.5) -- cycle;

\fill[line width=0pt,color=black,fill=ffzzcc!30] (-0.5,1.5) -- (-0.5,-1) -- (0.5,1) -- (1.5,1) -- (1.5,1.5) -- cycle;
\fill[line width=0pt,color=ffzzcc,fill=ffzzcc!] (-1.5,1.5) -- (-0.5,1.5) -- (1,0) -- (-0.5,1) -- (-1.5,1) -- cycle;

\fill[line width=0pt,color=ffzzcc,fill=ffzzcc] (0.5,1.5) -- (1.5,1.5) -- (1.5,1) -- (0.5,1) -- cycle;
\fill[line width=0pt,color=ffzzcc,fill=ffzzcc]  (0.5,-1.5) -- (1.5,-1.5) -- (1.5,-1) -- (0.5,-1) -- cycle;
\fill[line width=0pt,color=ffzzcc,fill=ffzzcc] (-0.5,1.5) -- (-1.5,1.5) -- (-1.5,1) -- (-0.5,1) -- cycle;
\fill[line width=0pt,color=ffzzcc,fill=ffzzcc]  (-0.5,-1.5) -- (-1.5,-1.5) -- (-1.5,-1) -- (-0.5,-1) -- cycle;

\draw [line width=1pt] (1,0)-- (-1,0);
\draw [fill=black] (-1,0) circle (2pt) node[below left] {\large$x$};
\draw [fill=black] (1,0) circle (2pt) node[below right] {\large$y$};

\draw [line width = 0.75pt, decorate, decoration={brace, amplitude=5pt}] (-1.55, -1)--(-1.55,1.5);

\draw [line width = 0.75pt, decorate, decoration={brace, amplitude=5pt}] (1.55,1.5)--(1.55, -1);

\node[left] at (-1.65, 0.25) {$A'$};
\node[right] at (1.65, 0.25) {$B'$};

\end{tikzpicture}
\caption{Illustration of the argument that $\gamma_A \approx \gamma_B$ in \cref{lem:two-positive-cliques}. We find $y \in B'$ and then $x \in A'$ such that $\gamma_A \approx \gamma_y \approx \ang{x,y} \approx \gamma_y \approx \gamma_B$.}
\end{figure}
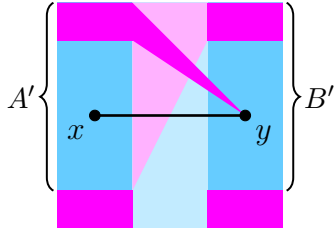

\begin{proof}[Proof of \cref{thm:structure} (Global structure theorem)]
    Fix $P = P(L)$ to be a large constant. Let $C$ be an $L$-code in $\RR^d$. By  \cref{lem:negative-clique}, the Gram graph $G_C$ cannot have a negative clique of size $\Delta_1 \coloneqq \lceil\beta^{-2}\rceil + 2$. Let us run the following algorithm.
\begin{itemize}
    \item Start with $C_0 = C$ and $k=0$;

   \item While $|C_k| \geq \operatorname{RamseyNumber}(P, \Delta_1)$,
   select $A_{k+1} \subseteq C_k$ to be a positive clique of size $P$. Apply \cref{lem:positive-clique} to partition $C_k \setminus A_{k+1} = B_{k+1} \sqcup C_{k+1}$ such that:
        \begin{itemize}
            \item The maximum negative degree from $B_{k+1}$ to $A_{k+1}$ is at most $\Delta_2$, and
            \item The maximum positive degree from $C_{k+1}$ to $A_{k+1}$ is at most $\Delta_2$.
        \end{itemize}
    Increase $k$ by $1$ and iterate.
\end{itemize}

\begin{figure}[h]\label{fig:structure-proof}
    \centering

\definecolor{wwccff}{rgb}{0.4,0.8,1}
\begin{tikzpicture}

\fill[line width=0pt,color=wwccff,fill=wwccff] (-2.5, 1.5) rectangle (-1.5, -1.5);
\node[below right] at (-1.6, -0.2) {$V_1'$};
\node at (-2, 1) {$A_1$};
\node at (-2, -0.75) {$B_1$};
\node at (-2, 0.3) { $+$};
\draw[black] (-2.4, 1.4) rectangle (-1.6, 0.6);
\draw[black] (-2.4, 0) rectangle (-1.6, -1.4);

\fill[line width=0pt,color=wwccff,fill=wwccff] (0, 1.5) rectangle (1, -1.5);
\node[below right] at (0.9, -0.2) {$V_2'$};
\node at (0.5, 1) {$A_2$};
\node at (0.5, -0.75) {$B_2$};
\node at (0.5, 0.3) { $+$};
\draw[black] (0.1, 1.4) rectangle (0.9, 0.6);
\draw[black] (0.1, 0) rectangle (0.9, -1.4);

\fill[line width=0pt,color=wwccff,fill=wwccff] (4.5,1.5) -- (3.5,1.5) -- (3.5,-1.5) -- (4.5,-1.5) -- cycle;
\node[below right] at (4.4, -0.2) {$V_q'$};
\node at (4, 1) {$A_q$};
\node at (4, -0.75) {$B_q$};
\node at (4, 0.3) { $+$};
\draw[black] (3.6, 1.4) rectangle (4.4, 0.6);
\draw[black] (3.6, 0) rectangle (4.4, -1.4);

\draw[black] (-3.2, 1.8) rectangle (6.55, -1.8);
\node[above] at (-2.85, -1.8) {$C_0$};
\draw[black] (-0.7, 1.7) rectangle (6.45, -1.7);
\node[above] at (-0.35, -1.7) {$C_1$};
\draw[black] (1.75, 1.6) rectangle (6.35, -1.6);
\node[above] at (2.1, -1.6) {$C_2$};
\draw[black] (5.25, 1.5) rectangle (6.25, -1.5);
\node[above] at (5.6, -1.5) {$C_q$};

\fill[line width=0pt,color=wwccff,fill=wwccff!40] (-1.5, 1.5) -- (-1.5, 0.5) -- (-0.7, -1.7) -- (-0.7, 1.7) -- cycle;
\node at (-1.1, 1) {$-$};
\fill[line width=0pt,color=wwccff,fill=wwccff!40] (1, 1.5) -- (1, 0.5) -- (1.75, -1.6) -- (1.75, 1.6) -- cycle;
\node at (1.4, 1) { $ - $};
\fill[line width=0pt,color=wwccff,fill=wwccff!40] (4.5, 1.5) -- (4.5, 0.5) -- (5.25, -1.5) -- (5.25, 1.5) -- cycle;
\node at (4.9, 1) { $-$};
\node at (2.6, 0) {$\cdots$};

\end{tikzpicture}
    \caption{The output of the algorithm in the proof of \cref{thm:structure}. Each $A_i$ is a positive clique of size $P$. $B_i$ is the set of vertices in $C_{i-1}$ such that nearly all edges between $B_i$ and $A_i$ are positive. $C_{i}$ is the set of remaining vertices in $C_{i-1}$.}
\end{figure}
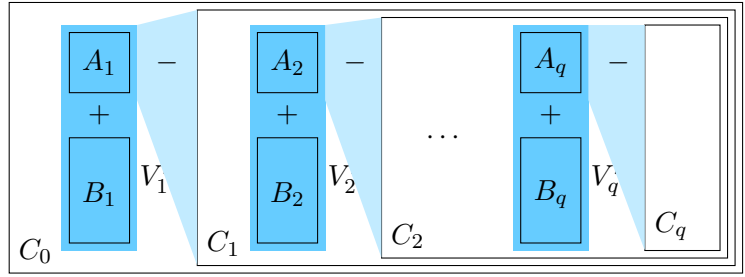
    Suppose the algorithm terminates with $k = q$. Each $A_i$ is a positive clique of size $P$ in $G_C$, and for all $i < j$ the maximum positive degree from $A_j$ to $A_i$ is at most $\Delta_2$. By \cref{lem:many-cliques}, $q \leq p$ as long as $P$ is large enough.

We will next prove that in each $V'_i \coloneqq A_i \cup B_i$, the maximum negative degree is $O_L(1)$. For $v \in V'_i$, let $S_v$ be set of negative neighbors of $v$ in $V'_i$. By \cref{lem:strong-positive-clique} there either $|A_i \setminus S_v| = O_L(1)$ or  $|S_v| = O_L(1)$.  By construction $|S_v \cap A_i| \leq \Delta_2$, so $|A_i \setminus S_v| \geq P - \Delta_2$ forces $|S_v| = O_L(1)$ as long as $P$ is large enough. This shows that the maximum negative degree in $V_i'$ is at $O_L(1)$.

By construction, for all $1 \leq i < j \leq q$, each $v \in V'_j$ has at least $P - \Delta_2$ negative neighbours in $V_i'$. By \cref{lem:positive-clique}, as long as $P$ is large enough, the maximum positive degree from $V_j'$ to $V_i'$ is $O_L(1)$.

 By \cref{lem:two-positive-cliques}, as long as $P$ is large enough, the following holds whenever $\Delta' \geq P$.  For any pair of distinct parts $\{V_i', V_j'\}$ with $|V_i'| \geq (\Delta')^2$ and $|V_j'| \geq (\Delta')^2$, we can find $ \gamma_{ij}' \in \RR$ such that for each $v \in V_i' \setminus R$, all but at most $\Delta'$ elements $u \in V_j'$ satisfy
 \[
 	|\langle u,v \rangle - \gamma_{ij}'| \leq P(\Delta')^{-1/2}.
 \]
Where $|R'| = O_L(\Delta')$. We will now remove small $V_i'$, $R'$, and $C_q$ from the graph. More precisely, define $R$ and $V_i$ by 
\[
    R = \paren{\bigcup_{|V_i'| < (\Delta')^2} V_i'} \cup R'  \cup C_q  \; \text{ and }  \; V_i = V_i' \setminus R.
\]
Note that $|R| = O_L(\Delta')^2$ and each $V_i$ is either empty or $|V_i| \geq \Omega_L(\Delta')^2$. To replace $\gamma_{ij}'$ by elements of $L \setminus \{\alpha\}$, for each pair non-empty pair $\{V_i, V_j\}$ fix some $u \in V_i$ and $v \in V_j$  so that $\langle u,v \rangle < 0$. Then,
\[
	|\langle u,v \rangle - \gamma_{ij}'| \leq P(\Delta')^{-1/2}.
\]
Set $\gamma_{ij} = \langle u,v \rangle \in L \setminus \{\alpha\}$. If any either $V_i$ or $V_j$ is empty, set $\gamma_{ij} \in L \setminus \{\alpha\}$ arbitrarily.

We have now accomplished the following: after removing at most $O_L(\Delta')^2$ elements from $C$, the rest can be partitioned as $ V_1 \sqcup \cdots \sqcup V_q$ with $q \leq p$ so that
\begin{itemize}
	\item[-] For $1 \leq i \leq q$ the maximum negative degree in $V_i$ is at $O_L(1)$.
	\item[-] For $i,j \in [q]$ and $i\ne j$, we can find $\gamma_{ij} \in L \setminus \{\alpha\}$ such that for each $u \in V_i$, for all but at most $\Delta'$ vectors $v \in V_j$ we have that $|\langle u, v \rangle - \gamma_{ij}| = O_L(\Delta')^{-1/2}$.
\end{itemize}
The conclusion follows immediately.
\end{proof}

The following corollary of \cref{thm:structure} will be used towards proving \cref{thm:main}.
A similar result appeared in the argument by Balla, Dr\"{a}xler, Keevash, Sudakov \cite{BDKS18}.

\begin{corollary}\label{cor:global} 
Fix $\alpha, \beta \in (0,1) $ and $L \subseteq [-1,-\beta]\cup\{\alpha\}$ and $p = \lfloor \alpha/ \beta \rfloor + 1$. There exists $\Delta = \Delta(\alpha,\beta)$  such that from any $L$-code, we can remove at most $\Delta$ vectors so that the negative graph of the remaining vertices is a $\Delta$-modification of complete $p$-partite graph.
\end{corollary}

\section{Rank Bound Within Each Part}\label{sec:rank-bound}

The global structure theorem partitions the $L$-code (minus $O_L(1)$ vectors) into at most $\floor{\alpha/\beta} + 1$ parts $V_1, \dots, V_q$.
Let us now examine the structure inside each $V_i$. 
We already know from the global structure theorem that the maximum negative degree inside each $V_i$ is $O_L(1)$. 
Our goal in this section is to show that, in order for the uniform bound to be tight, the negative graph in each $V_i$ must essentially be a perfect matching. 
For the following statements, one should think about vectors inside a single $V_i$.

Now we define a property enjoyed by negative components inside $V_i$ for an optimal modular code.

\begin{definition}\label{def:a-singular}
Let $ \alpha, \beta \in (0,1)$ and $L \subseteq [-1,-\beta] \cup \{\alpha\}$ with $\alpha \in L$. 
We say that a spherical $L$-code $C$ is \textit{$\alpha$-singular} if $M_C- \alpha J$ is a singular positive semidefinite matrix.
\end{definition}

\begin{proposition}\label{prop:a-singular}
Let $\alpha, \beta \in (0,1)$ and $L = [-1,-\beta] \cup \{\alpha\}$ with $\alpha \in L$.
Let $\lambda = (1-\alpha)/(\alpha + \beta)$. 
There exists an $\alpha$-singular $L$-code if and only if $\lambda \ge 1$. Furthermore, when $\lambda \ge 1$, there exists a unique $\alpha$-singular $L$-code of size two, namely the code with Gram matrix
\[
P = \begin{pmatrix}
            1 & 2\alpha - 1 \\
            2\alpha - 1 & 1
        \end{pmatrix}.
\]
\end{proposition}

\begin{proof}
    First, suppose $\lambda \geq 1$, or equivalently $2\alpha - 1\leq -\beta$. Then, the above $P$ is positive semidefinite and thus is the Gram matrix of an $L$-code. Furthermore, $P - \alpha J$ is singular and positive semidefinite.
    It is also not hard to see every $\alpha$-singular $L$-code of size two has to have $P$ as its Gram matrix, since the singularity condition forces the off-diagonal entry of $P$ to be $2\alpha - 1$ or $1$, but $1$ is not allowed off-diagonal.

    Next, suppose $\lambda < 1$, or equivalently $2\alpha - 1 > -\beta$. For contradiction, suppose there exists an $\alpha$-singular $L$-code with Gram matrix $M$. Since $M - \alpha J$ is singular, $M - \alpha J \neq (1-\alpha) I$. Thus, $M_{ij} < 0$ for some $(i,j)$. 
    So
    \[
        (e_i + e_j)^\intercal (M - \alpha J) (e_i + e_j) = 2 + 2M_{ij} - 4\alpha \leq 2 - 2\beta - 4\alpha < 0,
    \]
    which contradicts $M - \alpha J$ being positive semidefinite.
\end{proof}

\begin{proposition}[Rank bound]\label{prop:rank-bound}
Let $\alpha, \beta \in (0,1)$ and $L \subseteq [-1,-\beta] \cup \{\alpha\}$ with $\alpha \in L$.
Suppose $V$ is an $L$-code in $\RR^d$ with maximum negative degree at most $\Delta$.
Also suppose that among the connected components of its negative graph, exactly $r$ are $\alpha$-singular. 
Then
\[
\abs{V} \leq \begin{cases}
					d + O_{L,\Delta}\paren{\frac{d}{\log \log d}} & \text{ if } r = 0, \\
					d + r - 1 & \text{ if } r \neq 0.
				\end{cases}	
\]
\end{proposition}

\begin{proof}
Let $M_V$ be the Gram matrix of $V$. It can be written as
			\[
				M_V = (1-\alpha)I + \alpha J - (\alpha + \beta) A_H,
			\]
		where $A_H$ is the weighted adjacency matrix of a weighted graph $H$ on vertex set $V$ such that edge $uv$ has weight
			\[
				\frac{\alpha -\langle u, v \rangle}{\alpha + \beta} \in [1, 1 + \nu] \quad \text{ where } \, \nu \coloneqq \frac{1 - \beta}{\alpha + \beta}.
			\]
		Let $\mu = \alpha / (\alpha + \beta)$ and $\lambda = (1-\alpha)/(\alpha + \beta)$. Then
			\[
				(\alpha + \beta)^{-1}M_V = \lambda I - A_H + \mu J
			\]
		is positive semidefinite and has rank at most $d$. 
  Let $H[S]$ be the subgraph of $H$ induced by $S\subseteq V$. Then, $S$ is $\alpha$-singular if and only if $\lambda =  \lambda_1(A_{H[S]})$, since
		\[
			(\alpha + \beta)^{-1}(M_S - \alpha J) = \lambda I - A_{H[S]}.
		\]
		We will now consider few cases depending on how $\lambda$ relates to $\lambda_1(A_H)$.

\textit{Case 1.} Suppose that  $\lambda$ is not an eigenvalue of $A_H$. Then  $\lambda I - A_H$ has full rank. Hence,
				\[
					d \geq \mathrm{rk}(\lambda I - A_H + \mu J) \geq \mathrm{rk}(\lambda I - A_H) - 1 = |V| - 1.
				\]
			In this case $|V| \leq d + 1$, so we are done as no component is $\alpha$-singular.
						
\textit{Case 2.} Suppose $\lambda = \lambda_1(A_H)$. Let $S_1, S_2, ..., S_r$ be the vertex sets of the connected components of $H$ that form $\alpha$-singular  $L$-codes. By the Perron-Frobernius theorem, for all $j \in [r]$ we have $\lambda_2(A_{H[S_j]}) < \lambda_1(A_{H[S_j]}) = \lambda$. Furthermore, for any other connected component $S \subseteq V$, we have $\lambda_1(A_{H[S]}) < \lambda$. Therefore, 
				\[
					\dim \ker(\lambda I - A_H) = r.
				\]
			Since $\lambda I - A_H$ and $J$ are positive semidefinite, 
				\[
					\ker(\lambda I - A_H + \mu J) = \ker(\lambda I - A_H) \cap \ker J \le \ker(\lambda I - A_H) - 1 = r-1
				\]
    where the final inequality is due to the existence of a coordinate-wise nonnegative eigenvector of $A_H$ (again by the Perron-Frobernius theorem).
    Thus, by the  rank-nullity theorem,
    \[
    \abs{V} = \operatorname{rk}(\lambda I - A_H + \mu J) + \dim\ker (\lambda I - A_H + \mu J) \le d+ r-1.
    \]

\textit{Case 3.} Suppose  now that $\lambda = \lambda_2(A_H)$. Let $S_1, \ldots, S_k$ be the vertex sets of connected components of $G$ such that $\lambda_1(A_{H[S_1]}) \geq \ldots \geq \lambda_1(A_{H[S_k]})$. We first prove $\lambda_1(A_{H[S_j]}) < \lambda$ for all $1 < j \leq k$. 
			
			Suppose $\bm u$ and $\bm v$ are top eigenvectors of $A_{H[S_1]}$ and $A_{H[S_j]}$ respectively. By the Perron-Frobernius theorem, $\bm u$ and $\bm v$ have non-negative entries supported on corresponding components. In particular, $\bm u$ and $\bm v$ have disjoint supports.

			Since $\bm 1^\intercal \bm u > 0$ and $\bm 1^\intercal \bm v > 0$, there is $t \neq 0$ such that $\bm w =\bm  u + t\bm v$ and $\bm 1^\intercal  \bm w = 0$. Therefore, as $\lambda I - A_H + \mu J \succeq 0$, we have that
				\[
					0 \leq \bm w^\intercal(\lambda I - A_H + \mu J) \bm w = [\lambda - \lambda_1(S_1)]\bm u^\intercal \bm u + t^2[\lambda - \lambda_1(S_i)] \bm v^\intercal \bm v.
				\]
			Since $\lambda < \lambda_1(A_{H[S_1]})$ we conclude that $\lambda > \lambda_1(A_{H[S_i]})$.
			
			Now, by assumption, $G$ has maximum degree bounded by $\Delta$. In particular, $S_1$ is connected weighted graph with maximum degree bounded above by $\Delta$ with weights contained in $[1, 1+\nu]$. By the eigenvalue multiplicity bound \cref{thm:eigen-mult},
				\[
					\mathrm{mult}(\lambda, G) = \mathrm{mult}(\lambda, S_1) = O_{L,\Delta}\left(\frac{|S_1|}{\mathrm{loglog} |S_1|}\right) = O_{L,\Delta}\left(\frac{|V|}{\mathrm{loglog} |V|}\right).
				\]
So
				\[
					d \geq \mathrm{rk}(\lambda I - A_H + \mu J) \geq \mathrm{rk}(\lambda I - A_H) - 1 \geq |V| - O_{L,\Delta}\left(\frac{|V|}{\mathrm{loglog} |V|}\right).
				\]
			By rewriting, we obtain $|V| \leq d + O_{L,\Delta}\left(d/ \log \log d\right)$. Since  $\lambda < \lambda_1(A_{H[S_j]})$ for all $j \in [k]$, no component in $V$ is $\alpha$-singular.

\textit{Case 4.} Suppose $\lambda < \lambda_2(A_{H})$. Let $\bm u$ and $\bm v$ be two independent eigenvalues of $A_H$ each corresponding to an eigenvalue greater than $\lambda$.
Let $\bm w$ be a nonzero vector in the span of $\bm u$ and $\bm v$ and orthogonal to the all-1s vector.
Since $\lambda I - A_H + \mu J$ is positive semidefinite,
				\begin{align*}
					0\leq  \bm w^\intercal (\lambda I - A_H + \mu J)\bm w 
     = \bm w^\intercal (\lambda I - A_H)\bm w  
     \leq [\lambda - \lambda_2(G)]\bm w^\intercal \bm w 
     < 0.
				\end{align*}
			This is a contradiction.
\end{proof}

Roughly speaking, the rank bound shows that when the size of $V$ is large, there must be many $\alpha$-singular components present. The following corollary provides an upper bound on the size of certain $L$-codes in terms of the number of $\alpha$-singular components of size two that they contain.

\begin{corollary}\label{cor:rank-appl}
Let $\alpha, \beta \in (0,1)$ and $L \subseteq [-1,-\beta] \cup \{\alpha\}$. For all positive integers $p$ and $\Delta$ there exists $d_0 = d_0(L,p,\Delta)$ such that for all $d \geq d_0$ the following holds.

Suppose $C = V_1 \sqcup \cdots \sqcup V_p$ is an $L$-code in $\RR^d$ such that the maximum negative degree in $V_i$ is at most $\Delta$ for each $i \in [p]$. For each $i \in [p]$, let $r_i$ be the number of connected components of $G_{V_i}^-$ that form $\alpha$-singular $L$-codes of size two. Then
    \[
      |C| \leq \frac{3p(d-1)}{2} + \frac{(r_1 + \cdots + r_p)}{2}.
    \]
\end{corollary}

\begin{proof}
		 The number of components of $G_{V_i}^{-}$ that form  $\alpha$-singular codes of size at least three is at most $(|V_i| - 2r_i)/3$. Since $G_{V_i}^-$ has maximum degree at most $\Delta$, we can apply the rank bound, \cref{prop:rank-bound}, to see that
			\[
				|V_i| \leq \begin{cases}
				    d + r_i + \frac{|V_i| - 2r_i}{3} - 1 &\text{ if } r_i \neq 0, \\
                    d + O_{L,\Delta}\paren{\frac{d}{\log\log d}} &\text{ if } r_i = 0. \\
				\end{cases} 
			\]
		In particular, $|V_i| \leq \frac{3}{2}(d-1) + \frac{1}{2}r_i$ for all sufficiently large $d$. Summing up over all indices $ 1 \leq i \leq p$ concludes the proof.
\end{proof}

\section{Upper Bounds on Code Size}\label{sec:upper}

In this section we prove the upper bounds in the main results, \cref{thm:general-upper,thm:detail}. 

\begin{proof}[Proof of \cref{thm:general-upper}]
    By  \cref{cor:global}, there exists $\Delta = O_L(1)$ such that from any $L$-code we can remove at most $\Delta$ vectors so that the negative graph of the remaining vectors is a $\Delta$-modification of a complete $p$-partite graph with parts $V_1 \cup \cdots \cup V_p$. 

    Note that by \cref{prop:a-singular} there exists an $\alpha$-singular $L$-code if and only if $\lambda \geq 1$. Moreover, when $\lambda \geq 1$, the negative graph $G_{V_i}^-$ has at most $|V_i|/2$ connected components whose vertex sets form $\alpha$-modular codes. Therefore, by \cref{prop:rank-bound} applied to $V_i$ we see that for all sufficiently large $d$
    \[
        |V_i| \leq \begin{cases}
                    2d - 2 & \text{ if } \lambda \geq 1, \\
                    d + o(d) & \text{ if } \lambda < 1.
                    \end{cases}
    \]
    The result follows by summing over all $i \in [p]$ along with the $O_L(1)$ initially removed vertices.
\end{proof}

In the remainder of this section, we prove the upper bound of \cref{thm:detail}. Let us first make the following observation about how $\alpha$-singular codes of size two interact with each other.

\begin{lemma}\label{lem:a-singular-edges} Let $\alpha, \beta \in (0,1)$ and $L = [-1, -\beta] \cup \{\alpha\}$. Suppose $C$ is an $L$-code such whose negative graph is a perfect matching, consisting of edges edges $a_ia_i'$ for $i \in [r]$, with $r = \abs{C}/2$, and such that each $\set{a_i,a_i'}$ is an $\alpha$-singular $L$-codes of size two. Then, 
\[
    \frac{ a_i + a'_i}{2} =  \frac{ a_j + a'_j}{2} \quad \text{ for any} \quad  i, j \in [r].
\]
Moreover, writing their common midpoint as $ a = (a_i + a'_i) / 2$,
\[
    \left\{ \frac{ a_1 -  a}{\sqrt{1-\alpha}}, \cdots,  \frac{ a_r -  a}{\sqrt{1-\alpha}}\right\}
\]
is orthonormal and all orthogonal to $a$.   
\end{lemma}

\begin{proof}   
    Since the only edges in $G_C^-$ are formed by $\alpha$-singular $L$-codes $\{ a_i, a'_i\}$, all pairwise inner products in $C$ are equal to $\alpha$ except $\langle  a_i, a'_i \rangle = 2\alpha -1 $.
    Hence, for any distinct $i, j \in [n]$ 
    \begin{align*}
        \|( a_i + a'_i) - ( a_j + a'_j)\|^2 &= \|a_i + a_i'\|^2 + 2 \ang{a_i + a_i', a_j + a_j'} + \|a_j + a_j'\|^2 \\
        &=\paren{2 + 2(2\alpha - 1)} + 2 \paren{-4\alpha} + \paren{2 + 2(2\alpha - 1)} = 0.
     \end{align*}
     This shows that the pairs $\{ a_i, a'_i\}$ have a common midpoint $ a$. Similarly, we can compute  
    \[
        \langle  a_i -  a,  a_j -  a\rangle = \begin{cases}
            1-\alpha &\text{ if } i = j, \\
            0 &\text{ if } i \neq j.
                                                 \end{cases}
    \]
	The conclusion follows immediately. 
\end{proof}

The next result proves \cref{thm:detail}\ref{itm:detail-1}\ref{itm:detail-1non} and \cref{thm:detail}\ref{itm:detail-2}\ref{itm:detail-2non}.

\begin{proposition}\label{prop:algebraic-upper}
    Let $ \alpha, \beta \in (0,1)$ and $\lambda = (1-\alpha)/(\alpha + \beta)$ and $L = [-1, -\beta]\cup\{\alpha\}$. Suppose $\alpha/\beta + 1 = p \in \ZZ$.  
    \begin{enumerate}[(a)]
        \item\label{itm:algebraic-upper-2} If $p = 2$ and $\lambda^2 \notin \ZZ$, then 
        \[
            N_L(d) \leq \frac{11d}{3} + O_L(1),
        \]
        \item\label{itm:algebraic-upper-p} If $p \geq 3$ and $\lambda \notin \ZZ$, then 
        \[
            N_L(d) \leq \left(2p-\frac{1}{2}\right)d + O_L(1).
        \]
    \end{enumerate}
    \end{proposition}
   
\begin{proof}
    By  \cref{cor:global}, there exists $\Delta = O_L(1)$ such that from any $L$-code, we can remove at most $\Delta$ vectors so that the associated negative graph of the remaining vectors is a $\Delta$-modification of the complete $p$-partite graph on $V_1 \cup \cdots \cup V_p$. 
    
    Let $U_i \subseteq V_i$ be the union of connected components of $G_{V_i}^-$ that form $\alpha$-singular $L$-codes of size two. 
    Then the negative graph on $U_i$ is a perfect matching with edges $\{a_{ij}a_{ij}' : 1 \le j \le r_i\}$. Each pair $\{ a_{ij}, a'_{ij}\}$ forms an $\alpha$-singular $L$-codes for all $1 \leq i \leq p$ and $1 \leq j \leq r_i$. We can apply \cref{cor:rank-appl} to $V_1 \cup \cdots \cup V_p$ to see that for all sufficiently large $d$,
    \begin{equation}\label{eq:r-bound}
        |V_1 \cup \cdots \cup V_p| \leq \frac{3p(d-1)}{2} + \frac{(r_1 + \cdots + r_p)}{2}. \tag{$\dagger$}
    \end{equation}
    Therefore, it suffices to bound $r_1 + \cdots + r_p$ for all sufficiently large $d$.

     By \cref{lem:a-singular-edges} for all $1 \leq i \leq p$ there exists $a_i$ such $a_i = (a_{ij} + a_{ij}')/2$ for all $1 \leq j \leq r_i$  and the set of vectors
\[
    \left\{ \frac{ a_{i1} -  a_i}{\sqrt{1-\alpha}}, \cdots,  \frac{ a_{ir_i} -  a_i}{\sqrt{1-\alpha}}\right\}
\]
    is orthonormal and span a space orthogonal to $a_i$. In particular  $r_i \leq d-1$.
    
    Note that if $r_{i} \leq 16p^2\Delta$ for any $1 \leq i \leq p$ then we have that
    \[
        r_1 + \cdots + r_p \leq (p-1)(d-1) + 16 p^2 \Delta
    \]
    in which case we get $N_L(d) \leq (4p - 1)d/2 + O_L(1)$ from \cref{eq:r-bound}, and so the conclusion follows as $p \geq 2$.

    Suppose $r_i > 16p^2\Delta$ for all $1 \leq i \leq p$. Since $U_i \subseteq V_i$, the negative graph on $U_1 \cup \cdots \cup U_p$ is a $\Delta$-modification of the complete $p$-partite graph. Therefore, by iteratively applying the pigeonhole principle, we can find a two-element $\alpha$-singular component $s_i \subseteq U_i$ for each $i \in [p]$, such that for all $1 \leq i< j\leq p$, any edge with one endpoint in $s_i$ and other endpoint in $s_j$ is negative. Thus, $\ang{a_i, a_j} \leq -\beta$. We also have that $\ang{a_i, a_i} = \alpha$ (see \cref{prop:a-singular}). Therefore,
    \begin{align*}
        \left\|\sum_{i = 1}^{p}a_i\right\|^2 &= \sum_{j = 1}^{p}\ang{a_i,a_i} + 2\sum_{1 \leq i < j \leq p} \ang{a_i,a_j} \leq  \sum_{j = 1}^{p}\alpha + 2\sum_{1 \leq i < j \leq p} (-\beta)= p\alpha-p(p-1)\beta = 0.
    \end{align*}
    So the equality must hold at the inequality step. That is  $\ang{a_i, a_j} = -\beta$ for all $1 \leq i < j \leq p$. This allows us to exactly determine allowed inner products in $U_1 \cup \cdots \cup U_p$.  Suppose $1 \leq i < k \leq p$ and $j \in [r_i]$ and $t \in [r_k]$. Then 
    \begin{enumerate}
        \item Suppose all inner products between $\{a_{ij}, a_{ij}'\}$ and $\{a_{kt}, a_{kt}'\}$ are negative. Then
        \[
            -4\beta = \ang{2a_i, 2a_k} = \ang{a_{ij} + a_{ij}',a_{kt} + a_{kt}'} \leq -4\beta.
        \]
        Therefore, all inner products are equal to $-\beta$.
        \item Suppose there is at least one positive inner product between $\{a_{ij}, a_{ij}'\}$ and $\{a_{kt}, a_{kt}'\}$. Without loss of generality, assume that $\ang{a_{ij},a_{kt}} = \alpha$. 

        Note that by the pigeonhole principle, there exists $s \in [r_k]$ such that all inner products between $\{a_{ij}, a_{ij}'\}$ and $\{a_{ks}, a_{ks}'\}$ are negative. Hence
        \[
            \ang{a_{ij}, 2a_k} =   \ang{a_{ij}, a_{ks}+a_{ks}'} = -2\beta.
        \]
        Analogously, we have that $\ang{a_{kt},2a_i} = -2\beta$. Therefore, we have that
        \begin{align*}
            \ang{a_{ij}, a_{kt}'} &=  \ang{a_{ij}, 2a_k - a_{kt}} = -(2\beta + \alpha), \\
            \ang{a_{ij}', a_{kt}} &=  \ang{2a_i - a_{ij}, a_{kt}} = -(2\beta + \alpha), \text{ and} \\
            \ang{a_{ij}', a_{kt}'} &=  \ang{2a_i - a_{ij}, 2a_k - a_{kt}} = -4\beta + 4\beta + \alpha = \alpha.
        \end{align*}
        Hence, the inner products between $\{a_{ij}, a_{ij}'\}$ and $\{a_{kt}, a_{kt}'\}$ are in $\{\alpha, -(2\beta + \alpha)\}$.
    \end{enumerate}
    For $1 \leq i \leq p$ and $1 \leq j \leq r_i$ take $u_{ij} = (a_{ij} - a_i)/\sqrt{1-\alpha}$. By \cref{lem:a-singular-edges} we know that for all $i \in [p]$ the set of vectors $\{u_{ij} : 1 \le j \le r_i\}$ is orthonormal. Furthermore, by the observation above, whenever $1 \leq i < k \leq p$ and $j \in [r_i]$ and $t \in [r_k]$ we have that
    \[
        \ang{u_{ij}, u_{kt}} = \frac{\ang{a_{ij}-a_i, a_{kt}-a_k}}{1-\alpha} = \frac{\ang{a_{ij}, a_{kt}} + \beta}{1-\alpha} \in \{\pm \lambda^{-1}, 0\}.
    \]
    Let $M$ be the Gram matrix of vectors $\{u_{ij} : i \in [p], \, j \in [r_i]\}$, which has block form
    \[
        M = \begin{pmatrix}
            I_{r_1} & \lambda^{-1}R_{12} & \cdots & \lambda^{-1} R_{1p} \\
            \lambda^{-1}R_{12}^\intercal & I_{r_2} & \cdots & \lambda^{-1} R_{2p} \\
            \vdots & \vdots & \ddots & \vdots \\
            \lambda^{-1}R_{1p}^\intercal & \lambda^{-1}R_{2p}^\intercal & \cdots & I_{r_p} \\
             \end{pmatrix},
    \]
    where $R_{ik} \in \{0,\pm 1\}^{r_i \times r_k}$ for $1 \leq i < k \leq p$. Let $\lambda M = \lambda I - Q$ where $Q$ is an integer matrix. If $\lambda$ is not an eigenvalue of $Q$ then $M$ is full rank. Thus, $d \geq\mathrm{rk}(M)= r_1 + \cdots + r_p$, so from \cref{eq:r-bound} we get $N_L(d) \leq (3p+1)d/2 + O_L(1)$. The conclusion follows as $p \geq 2$.

   We now assume that $\lambda$ is an eigenvalue of $Q$ and split into two cases based on whether $p \geq 3$.

    \textit{Case 1.} Suppose $p \geq 3$ and $\lambda \notin \ZZ$. Since $Q$ is an integer matrix, its characteristic polynomial is a monic integer polynomial with root $\lambda$. Hence, $\lambda \notin \QQ$, and so its minimal polynomial is at least quadratic. Thus the degree of the characteristic polynomial of $Q$ is at least twice the multiplicity of $\lambda$. Hence,
    \[
        d \geq \mathrm{rk}(M) = r_1 + \cdots + r_p - \mathrm{mult} \left(\lambda, Q\right) \geq \frac{r_1 + \cdots + r_p}{2}
    \]
    In particular $r_1 + \cdots + r_p \leq 2d$. Thus $N_L(d) \leq (3d + 2)d/2 + O_L(1)$ from \cref{eq:r-bound}, so the conclusion follows as $p \geq 3$.
    
    \textit{Case 2.} Suppose $p = 2$ and $\lambda^2 \notin \ZZ$. Let $R = -R_{12} \in \ZZ^{r_1 \times r_2}$. Then
    \[
        Q = \begin{pmatrix}
            0 & R \\
            R^\intercal & 0 \\
        \end{pmatrix}.
    \]
    We first show that $\mathrm{mult}(\lambda, Q) = \mathrm{mult}(\lambda^2, RR^\intercal)$. By a straightforward matrix arithmetics we can check that the following identity holds for all $t \neq 0$
    \[
         \begin{pmatrix}
            tI & -R \\
            -R^\intercal & tI \\
        \end{pmatrix} 
        \begin{pmatrix}
            tI & 0 \\
            R^\intercal & tI \\
        \end{pmatrix} 
        =\begin{pmatrix}
            I & -t^{-1}R \\
            0 & I \\
        \end{pmatrix} 
        \begin{pmatrix}
            t^2I - RR^\intercal & 0 \\
            0 & t^2I \\
        \end{pmatrix}.
    \]
    Taking the determinant of both sides, we get $t^{r_1 + r_2}\det \paren{t I - Q} = t^{2r_2}\det \paren{t^2 I - RR^\intercal}$ for all $t \neq 0$. Since $\lambda \neq 0$, we conclude that $\mathrm{mult}(\lambda, Q) = \mathrm{mult}(\lambda^2, RR^\intercal)$. Thus,
    \[
        d \geq \mathrm{rk}(\lambda M) = \mathrm{rk}(\lambda I - Q) = r_1 + r_2 - \mathrm{mult}(\lambda, Q) \geq r_1 + r_2 - \mathrm{mult}(\lambda^2, RR^\intercal).
    \]    
    Since $RR^\intercal \in \ZZ^{r_1 \times r_1}$ and $\lambda^2 \notin \ZZ$, we have $\mathrm{mult}(\lambda^2, RR^\intercal) \leq r_1/2$. Hence, $d \geq r_1 + r_2/2$.
    The same argument shows that $d \geq r_1/2 + r_2$. In particular, $ r_1 + r_2 \leq 4d/3$, so we get $N_L(d) \leq 11d/3 + O_L(1)$ from \cref{eq:r-bound}.
\end{proof}

The next result proves the upper bound in \cref{thm:detail}\ref{itm:detail-1}\ref{itm:detail-1int}.

\begin{proposition}\label{prop:upper}
    Let $ \alpha \in (0,1)$ and $L = [-1,-\alpha]\cup\{\alpha\}$. Then, there exists $d_0 = d_0(\alpha)$ such that for all $d \geq d_0$,
    \[
        N_L(d) \leq 4(d-1).
    \]
\end{proposition}

\begin{proof}
 Let $C$ be a spherical $L$-code in $\RR^d$. By the global structure theorem, \cref{cor:global}, there exists $J \subseteq C$ with $\abs{J} = O_\alpha(1)$ such that the associated negative graph of $C \setminus J$ is a $O_\alpha(1)$-modification of complete bipartite graph. Let $C \setminus J = V_1 \cup V_2$ be the bipartition.

By \cref{lem:positive-clique} for $i \in \{1,2\}$, each vertex $v \in J$ has either $O_\alpha(1)$ positive or $O_\alpha(1)$ negative edges to $V_i$. If $v$ has $O_\alpha(1)$ positive edges to $V_2$ and $O_\alpha(1)$ negative edges to $V_1$, then we can remove $v$ from $J$ and place it in $V_1$. Likewise with $V_1$ and $V_2$ swapped.
After removing such vertices out of $J$, every remaining $v \in J$ has either at most $\Delta$ positive edges to $V_1 \cup V_2$, or at most $\Delta$ negative edges to $V_1 \cup V_2$, for some constant $\Delta = \Delta(\alpha)$.

Let $r_i$ be the number of connected components of $G_{V_i}^-$ that form modular $L$-codes of size two. By \cref{cor:rank-appl}, for all sufficiently large $d$,
\[
    |V_1 \cup V_2| \leq 3(d-1) + \frac{r_1 + r_2}{2}.
\]
On the other hand, by \cref{lem:a-singular-edges} we have that $r_i \leq d-1$ and so $|V_1 \cup V_2| \leq 4(d-1)$. Hence, it suffices to show $J = \varnothing$.

Observe that if $r_1 \leq 4\Delta$ or $r_2 \leq 4\Delta$ then $|V_1 \cup V_2| \leq 3.5d + O_L(1)$, and so $\abs{C} \le 3.5d + O_L(1) \le 4(d-1)$ for sufficiently large $d$.
Hence, assume that $r_1 > 4\Delta$ and $r_2 > 4 \Delta$. Let $U_i \subseteq V_i$ be the subsets formed by the union of all connected components of $G_{V_i}^-$ that form $\alpha$-singular $L$-codes of size two. Suppose 
 \[
        U_i = \{a_{i1}, a'_{i1}\}\cup \cdots \cup \{a_{ir_i}, a'_{ir_i}\}.
\]  
We can apply \cref{lem:a-singular-edges} to $U_i$ and take $a_i =(a_{i1} + a'_{i1})/2$. By the pigeonhole principle, there exists $1 \leq j \leq r_2$ such that all four edges with one endpoint in $\{a_{11}, a'_{11}\}$ and the other in $\{a_{2j}, a'_{2j}\}$ are negative. Hence,
    \begin{align*}
        \| ( a_{11} + a'_{11}) + (a_{2j} + a'_{2j}) \|^2 &= \|a_{11} + a'_{11}\|^2 + \|a_{2j} + a'_{2j}\|+2\ang{   a_{11} + a'_{11}, a_{2j} + a'_{2j}}\\
        &= 8\alpha+2\ang{   a_{11} + a'_{11}, a_{2j} + a'_{2j}} \\
        &\leq 8\alpha - 8\alpha \leq 0.
    \end{align*}
    In particular that $a_1 + a_2 = 0$. 
    
    For each $ v \in J$, by the the pigeonhole principle there exist $i \in [r_1]$ and $j \in [r_2]$ such that all four inner products
\[
    \ang{v, a_{1i}}, \quad
    \ang{v, a'_{1i}}, \quad
    \ang{v, a_{2j}}, \quad
    \ang{v, a'_{2j}}
\]
all have the same nonzero sign. However, this is a contradiction as
\[
    \ang {v, a_{1i} + a'_{1i} + a_{2j} + a'_{2j} } = \ang {v, 2a_1+2a_2 } = 0.
\]
Thus, $J = \varnothing$. Hence, $|C| = |V_1| + |V_2| \leq 4(d-1)$. 
\end{proof}

\bibliographystyle{amsplain0}
\bibliography{bibliography}

\end{document}